\documentclass[reqno,a4paper,11pt]{amsart}
\usepackage[utf8]{inputenc}

\calclayout

\usepackage[T1]{fontenc}
\usepackage[utf8]{inputenc}

\usepackage{verbatim}
\usepackage[english]{babel}
\usepackage{amsmath}
\usepackage{amsfonts}
\usepackage{mathtools}
\usepackage{upgreek}
\usepackage{amsthm}
\usepackage{bm}
\usepackage{float}
\usepackage{color}

\usepackage[normalem]{ulem}

\usepackage{subcaption}
\usepackage{enumerate}
\usepackage{bm}
\usepackage{seqsplit}

\DeclareMathOperator{\Tr}{\mathrm{Tr}}

\DeclareMathOperator{\ai}{Ai}
\DeclareMathOperator{\re}{Re}
\DeclareMathOperator{\im}{Im}
\DeclareMathOperator{\ee}{\rm e}

\DeclareMathOperator{\supp}{supp}
\newcommand{\itid}{\mathsf{1}}  

\newcommand{\N}{\mathbb{N}}
\newcommand{\C}{\mathbb{C}}
\newcommand{\R}{\mathbb{R}}
\newcommand{\Z}{\mathbb{Z}}

\newcommand{\dd}{\mathrm{d}}
\newcommand*{\deff}{\mathrel{\vcenter{\baselineskip0.5ex \lineskiplimit0pt
                     \hbox{\scriptsize.}\hbox{\scriptsize.}}}%
                     =}
\newcommand*{\revdeff}{=\mathrel{\vcenter{\baselineskip0.5ex \lineskiplimit0pt
                     \hbox{\scriptsize.}\hbox{\scriptsize.}}}%
                     }

\newcommand\mb[1]{\mathbb{#1}}
\renewcommand{\bm}{\mathbf}
\newcommand{\mcal}{\mathcal}

\newcommand{\msf}{\mathsf}

\newcommand{\wh}{\widehat}

\newcommand{\Ai}{\mathrm{Ai}}

\newcommand{\dsin}{{\mathrm{dsin}}}

\newcommand{\K}{\mathbb K}

\newtheorem{theorem}{Theorem}[section]

\newtheorem{lemma}[theorem]{Lemma}
\newtheorem{corollary}[theorem]{Corollary}

\theoremstyle{definition}

\newtheorem{assumption}[theorem]{Assumption}

\theoremstyle{remark}
\newtheorem{remark}[theorem]{Remark}
\numberwithin{equation}{section}

\newcommand\restr[2]{{
		\left.\kern-\nulldelimiterspace 
		#1 
		\vphantom{\big|} 
		\right|_{#2} 
}}

\newcommand{\rev}{}

\begin{document}

\title{Deformations of biorthogonal ensembles and universality}

\author[T.~Claeys]{Tom Claeys}
\address[TC]{UCLouvain, Belgium.}
\email{tom.claeys@uclouvain.be}

\author[G.~Silva]{Guilherme L.~F.~Silva}
\address[GS]{Instituto de Ciências Matemáticas e de Computação, Universidade de S\~ao Paulo, Brazil.}
\email{silvag@icmc.usp.br}

\date{}


\begin{abstract}
We consider a large class of deformations of continuous and discrete biorthogonal ensembles and investigate 
their behavior in the limit of a large number of particles.
We provide sufficient conditions to ensure that if a biorthogonal ensemble converges to a (universal) limiting process, then the deformed biorthogonal ensemble converges to a deformed version of the same limiting process. To construct the deformed version of the limiting process, we rely on a procedure of marking and conditioning. Our approach is based on an analysis of the probability generating functionals of the ensembles and is conceptually different from the traditional approach via correlation kernels.
Thanks to this method, our sufficient conditions are rather mild and do not rely on much regularity of the original ensemble and of the deformation. 
As a consequence of our results, we obtain probabilistic interpretations of several Painlevé-type kernels that have been constructed in the literature, as deformations of classical sine and Airy point processes.
\end{abstract}


\vspace*{-1.6cm}

\maketitle

\section{Introduction}

The notion of {\em universality} is omnipresent in the modeling of physical phenomena with complex interactions. Informally speaking, it means that large families of models share the same or similar features on microscopic scales, or in other words that microscopic behavior is to a large extent independent of specific macroscopic dynamics of the model at hand {\rev\cite{DeiftUniversality2007}}.

Random matrix theory is among the richest grounds for the manifestation and deep understanding of universality, and it is understood that the microscopic behavior of eigenvalues in a large class of {\rev Hermitian} random matrix ensembles is described in terms of a limited number of universal point processes. This means that in the large matrix limit, the microscopic behavior of eigenvalues is described by one of a few canonical processes, for instance the sine point process in the bulk of the spectrum and the Airy point process at soft edges of the spectrum. In non-equilibrium statistical mechanics, large classes of growth models are described, in appropriate large time scales, by universal fluctuation fields, such as the KPZ fixed point or the Gaussian Free Field, which contain in appropriate senses several of the aforementioned random matrix limiting processes {\rev \cite{Kenyon2004, Remenik2023ICM, borodin_gorin_notes}}.

A rigorous mathematical justification of universality is often possible in models that have an integrable structure. In the aforementioned random matrix theory, for instance, a form of integrability arises at the level of eigenvalues, which in many models are described in terms of determinantal point processes. 
In such models, the relevant statistics are encoded in a function of two variables, the correlation kernel, which depends on the size of the system in a parametric way.
Through the analysis of this correlation kernel in the large size limit, one is able to leverage local information to provide universality results in its various different forms.

A natural viewpoint on universality is that {\rev perturbations or deformations} of the original system should not affect their large size microscopic behavior. While it is widely understood that microscopic properties are indeed little affected by macroscopic and mesoscopic deformations of models, we will show in contrast that they are sensitive to microscopic deformations. More precisely, we will show that such microscopic deformations lead to entire novel families of possible microscopic limit point processes. 

We consider a large family of determinantal point processes, known as biorthogonal ensembles, {\rev consisting of $n$ particles} whose correlation functions admit a representation in terms of biorthogonal functions. This family of processes encompasses several important random matrix models, certain families of growth processes, polymer models, random tilings, random partitions and more. We consider deformations of biorthogonal ensembles, and investigate under which conditions the behavior of the original biorthogonal ensemble carries over to the deformed ensemble, in the limit of a large number of particles. {\rev Here, we are concerned with the following question:} if the original biorthogonal ensemble satisfies some form of universality on microscopic scales, will the deformed biorthogonal ensemble satisfy the same type of universality, or will the deformation cause a different type of microscopic behavior?

For a class of deformations of the finite $n$ particle system which take place at a microscopic scale, we will show that in the large $n$ limit, the universal limiting point process is also deformed.  Remarkably, both the deformed biorthogonal ensemble and the deformed limiting point process can be constructed via a procedure of marking and conditioning. In this way, even if the limiting point process is deformed, our results show that it is still determined by the limit of the non-deformed biorthogonal ensemble, together with the form of the deformation. 

The main conceptual novelty is that we approach this problem via the probability generating functional of the point processes and not via the more commonly used scaling limits of correlation kernels. This different angle has two major advantages. Firstly, we do not need to perform a technical asymptotic analysis of the deformed models and we can instead rely on knowledge of the undeformed point process combined with properties of the deformation. Secondly, we only need minimal regularity requirements for the deformations.

We now move on to the detailed discussion of our results.

\section{Statement of results}

Throughout the paper, we use $\K$ to denote the state space, or in other words the space to which particles of the system belong. We will restrict ourselves to the case where $\mathbb K=\mathbb R^d$ with $d$ a positive integer, in particular also encompassing point processes on $\mathbb C^m\simeq \R^{2m}$ . Our methods could apply to a more general framework with $\K$ being a {\rev subspace} of an arbitrary finite-dimensional complex Banach space, but for the applications we have in mind the case $\K=\R^d$ suffices, and to avoid more technical arguments we restrict to them.

\subsection*{Biorthogonal ensembles}\hfill \\
A {\it biorthogonal ensemble} on $\K$ \cite{Borodin99biorth, KuijlaarsICM} is a probability {\rev measure} on $\K^n$, invariant under permutation of the variables, of the form
\begin{equation}\label{def:biOPE}
\frac{1}{\msf z_n}\det\left[f_k(x_j)\right]_{j,k=1}^{n}\det\left[{g_k(x_j)}\right]_{j,k=1}^{n} \prod_{j=1}^n w_n(x_j)\dd \nu(x_j).
\end{equation}
Here the positive reference measure $\nu$ is a Radon measure on $\K $, possibly varying with $n$, the systems of measurable and complex-valued functions $(f_j)_{j=1}^{n}, (g_j)_{j=1}^{n}$ are such that 
$$
\det\left[f_k(x_j)\right]_{j,k=1}^{n}\det\left[{g_k(x_j)}\right]_{j,k=1}^{n}> 0
\qquad \text{for } \nu\text{-a.e. } (x_1,\ldots, x_n)\in \K^n,
$$ 
the weight function $w_n:\K \to[0,+\infty)$ is $\nu$-integrable, and the normalization constant $\msf z_n$ is given by
\[\msf z_n\deff \int_{\K^n}\det\left[f_k(x_j)\right]_{j,k=1}^n\det\left[{g_k(x_j)}\right]_{j,k=1}^{n} \prod_{j=1}^{n}w_n(x_j) \dd \nu(x_j).\]
The entries of a random vector $(x_1,\ldots,x_n)$ distributed according to \eqref{def:biOPE} are seen as {\rev locations $x_j$ of random particles} in the state space $\K$, and we refer to \eqref{def:biOPE} as a biorthogonal ensemble with $n$ particles. Typical absolutely continuous choices of $\nu$ include the ($n$-independent) Lebesgue measure $\dd\nu(x)=\dd x$ on $\R$, $\R^2$ or $\C$, or also the Lebesgue measure restricted to the unit interval $[0,1]\subset \R$. In a similar spirit, in a discrete setup we can take $\nu=\nu_n$ to be the counting measure on the rescaled positive integer lattice $\frac{1}{n}\Z_{\geq 0}$, hence varying with the number {\rev $n$} of particles. 

The weight function, as well as the functions $f_k=f_{n,k}$ and $g_k=g_{n,k}$, may also depend on the number {\rev $n$} of particles, although we will not make such dependence explicit in our notation unless needed. Without loss of generality we could have omitted the weight $w_n$ in \eqref{def:biOPE}, 
since $w_n$ can be absorbed by the reference measure $\nu$ or by the functions $f_k, g_k$, but it will turn out instructive and convenient to separate the dependence on $w_n$ from the rest of the measure. 

Arguably the most studied of such models {\rev \eqref{def:biOPE}} are the {\em orthogonal polynomial ensembles} \cite{Konig05, KuijlaarsICM}, obtained when we take $\dd\nu(x)=\dd x$ and $f_k(x)=g_k(x)=x^{k-1}$ on $\K=\R$, such that both determinants are Vandermonde determinants and the distribution becomes
\begin{equation}\label{def:OPE}
\frac{1}{\msf z_n}\prod_{1\leq j<k\leq n}(x_k-x_j)^2 \prod_{j=1}^n w_n(x_j)\dd x_j.
\end{equation}
For $w_n=\ee^{-nV}$ and $V$ sufficiently regular and with sufficient growth at $\pm\infty$, this is the probability distribution of the eigenvalues in the {\em unitary invariant random matrix ensemble}
\begin{equation}\label{def:UE}
\frac{1}{\msf z_n'}\ee^{-n{\Tr V(M)}}\dd M,\qquad \dd M=\prod_{j=1}^n \dd M_{jj}\,\prod_{1\leq j<k\leq n}\dd\re M_{jk}\dd\im M_{jk},
\end{equation}
on the space of $n\times n$ Hermitian matrices $M$ \cite{BleherIts99,Deiftetal1999,Deiftetal99b}. 
Another interesting and more general situation occurs when $f_k(x)=x^{k-1}$ but with $g_k$ general, in which case we obtain the {\em polynomial ensemble} \cite{KuijlaarsStivigny14}
\begin{equation}\label{def:PE}
\frac{1}{\msf z_n}\prod_{1\leq j<k\leq n}(x_k-x_j)\det\left[g_k(x_j)\right]_{j,k=1}^n \prod_{j=1}^n w_n(x_j)\dd x_j.
\end{equation}
Probability distributions of this form occur for singular values of products of random matrices \cite{AkemannIpsen15, AkemannKieburgWei13, KuijlaarsStivigny14, KieburgKuijlaarsStivigny16}, coupled random matrix models \cite{bleher_kuijlaars_external_source_multiple_orthogonal, DuitsKuijlaars09} or matrix models with eigenvalues in the plane \cite{BaloghBertolaLeeMcLaughlin15,KuijlaarsBleher12},  and in Muttalib-Borodin ensembles \cite{Borodin99biorth}. The general family of biorthogonal ensembles is rather large, and particular instances occur in connection to non-intersecting random walks and paths, random growth models, random tilings, random partitions, last passage percolation and polymer models, among many others \cite{CafassoClaeys24, delvaux_kuijlaars_zhang,Duits18, DuitsKuijlaars2021, ImamuraSasamoto16}.

There is a natural discrete analogue to \eqref{def:PE}, obtained upon replacing $\dd x$ by the counting measure $\dd\nu$ of, say, the re-scaled integers $\frac{1}{n}\Z\subset \mb K=\R$, or another discrete lattice over $\R$. Such discrete measures form part of the class of discrete Coulomb gases \cite{GuionnetJiaoyang19, johansson_2003}, and besides their natural interest as gases in statistical mechanics, they also appear in connection with exclusion processes, the six-vertex model, asymptotic representation theory, random tilings, to mention only a few \cite{Johansson2001,Johansson2005RWHahn, Borodin2017, Bleher2011, bleher_liechty_book, borodin_gorin_notes}.

Biorthogonal ensembles on subsets of $\mathbb R^2$ (or equivalently $\mathbb C$) arise for non-Hermitian random matrices and biorthogonal ensembles on higher-dimensional spaces have been explicitly constructed lately, for instance the higher dimensional elliptic ensembles \cite{AkemannMolagDuits2023} and the spherical ensembles \cite{BeltranMarzoOrtega-Cerda2016}. Some related constructions of point processes on more general manifolds are also of interest (see for instance \cite{Berman2014} and the references therein) and many of the aspects we will discuss may be adapted to these frameworks as well with non-essential modifications, but for the sake of clarity we restrict to considering the state space $\K$ to be a Euclidean space.

\subsection*{Determinantal point processes.}\hfill \\
We will assume throughout the paper that the functions $f_j$ and $g_j$ defining the biorthogonal ensemble \eqref{def:biOPE} are such that $\sqrt{w_n}f_j$, $\sqrt{w_n}g_j$ belong to $L^2(\nu)= L^2(\K\to \mb C,\nu)$.
Then, the biorthogonal ensemble \eqref{def:biOPE} is a {\em determinantal point process} \cite{Borodin99biorth,Soshnikov2000a}:
the $m$-point correlation functions $\rho_{n,m}:\K^m\to [0,+\infty)$ exist for any $m\in\mathbb N$, and there exists a kernel $\msf k_n:\K^2\to \mathbb C\in L^2(\nu\times \nu)$ for which they take the form
\begin{equation}\label{def:DPPcorr}
\rho_{n,m}(x_1,\ldots, x_m)=\det\left(\msf k_n(x_j,x_k)\right)_{j,k=1}^m.
\end{equation}
The correlation kernel of a determinantal point process is not unique. However, for biorthogonal ensembles the kernel $\msf k_n$ may be taken in the form
\begin{equation}\label{def:kernel}
\msf k_n(x,y)=\sqrt{w_n(x)w_n(y)}\sum_{k=1}^{n}\phi_k(x){\psi_k(y)},
\end{equation}
where the {\rev (complex)} linear span of $(\phi_k)_{k=1}^{n}$ is equal to that of $(f_k)_{k=1}^{n}$, and the linear span of $(\psi_j)_{j=1}^{n}$ is equal to that of $(g_k)_{k=1}^{n}$, and they satisfy the {\em biorthogonality relations} that give name to the class,
\begin{equation}\label{eq:biorth}
\int_{\K}\phi_k(x){\psi_j(x)}w_n(x)\dd\nu(x)=\delta_{jk}.
\end{equation}
Note that the left hand side above is the inner product of $\phi_k$ with $\overline{\psi_j}$, not the inner product of $\phi_k$ with $\psi_j$.

From the general theory of determinantal point processes \cite{Johansson06,Borodin2011,Soshnikov2000a}, we then know that statistics of \eqref{def:biOPE} can be computed through a
{\em Fredholm series}: for any bounded measurable function $h:\K\to \R$, we have that
\begin{equation}\label{def:Fredholmseries}
\mathbb E_n\left[\prod_{k=1}^n(1-h(x_k))\right]= \sum_{k=0}^\infty\frac{(-1)^k}{k!}\int_{\K^k}\det\left(\msf k_n(x_j,x_\ell)\right)_{j,\ell=1}^k \prod_{j=1}^k h(x_j)\dd \nu(x_j),
\end{equation}
where the expectation $\mathbb E_n$ is with respect to the probability {\rev law} \eqref{def:biOPE}, and where all terms corresponding to $k>n$ on the right-hand side vanish.

Since  $\sqrt{w_n}\phi_j$, $\sqrt{w_n}\psi_j$ are in $L^2(\nu)$,  the integral operator
$$
\bm k_n f(x)=\int_{\K}\msf k_n(x,y){f(y)}\dd \nu(y)
$$
is a finite rank (not necessarily Hermitian) linear projection operator 
on $L^2(\nu)$, whose range is the linear span of $\sqrt{w_n}\phi_1,\ldots, \sqrt{w_n}\phi_{n}$, and whose kernel is the orthogonal complement of the linear span of $\sqrt{w_n}\overline{\psi_1},\ldots, \sqrt{w_n}\overline{\psi_{n}}$. 
In particular, $\mathbf k_n$ is a {\em trace-class operator}, and the right-hand side of \eqref{def:Fredholmseries} is then equal to the {\em Fredholm determinant} 
$\det(\itid-h\mathbf k_n)_{L^2(\nu)}$ defined through the functional calculus for trace-class operators \cite{Simon_trace}, where $h$ denotes the operator of multiplication with the function $h$. When the underlying measure is clear from the context or unimportant, we sometimes drop the notation $L^2(\nu)$ and write simply $\det(\itid-h\mathbf k_n)_{L^2(\nu)}=\det(\itid-h\mathbf k_n)$.
{\rev For a sufficiently large class of functions $h$, the expectations in \eqref{def:Fredholmseries} characterize the point process \cite{DaleyVere-Jones-VolI}.}

\subsection*{Scaling limits}\hfill \\
A common approach towards {\rev asymptotic} results in biorthogonal ensembles goes via {\em scaling limits of the correlation kernel}. Given a sequence of biorthogonal ensembles \eqref{def:biOPE} with correlation kernels $\msf k_n$, let us fix a reference point $x^*\in\K$, and define the re-scaled kernel
\begin{equation}\label{def:rescaled kernel}
{\msf K}_n(u,v)\deff \frac{1}{cn^\gamma}\msf k_n\left(x^*+\frac{u}{cn^\gamma},x^*+\frac{v}{cn^\gamma}\right),\quad u,v\in \K,
\end{equation}
for suitably chosen $c\in\mb C,\gamma>0$, possibly depending on the choice of $x^*\in \K$ and on the choice of $w_n$.
This kernel ${\msf K}_n$ also defines a scaled determinantal point process $\mathcal X_n$ which is a biorthogonal ensemble, namely the probability distribution 
\begin{equation}\label{def:BiOEscaled}
\frac{1}{\msf Z_n}\det\left[F_k(u_j)\right]_{j,k=1}^{n}\det\left[{G_k(u_j)}\right]_{j,k=1}^{n} \prod_{j=1}^n W_n(u_j)\dd \mu_n(u_j),\qquad u_1,\ldots, u_n\in\K,
\end{equation}
where
\begin{equation}\label{def:wtP}
F_k(u)=f_k\left(x(u)\right),\  G_k(u)=g_k\left(x(u)\right),\  W_n(u)=w_n\left(x(u)\right),\ \dd\mu_n(u)\deff
cn^\gamma \dd \nu\left(x(u)\right),
\end{equation}
with $x(u)=x^*+\frac{u}{cn^\gamma}$,
and 
\[
\msf Z_n\deff \int_{\K^n}\det\left[F_k(u_j)\right]_{j,k=1}^n\det\left[{G_k(u_j)}\right]_{j,k=1}^{n} \prod_{j=1}^{n}W_n(u_j) \dd \mu_n(u_j).
\]

This is in other words the probability distribution for the scaled points $u_j=u(x_j)=cn^\gamma(x_j-x^*)$ living on $\K$. When, for instance, $\dd\nu(x)=\dd x$ is the Lebesgue measure on $\R$, then $\dd \mu_n(u)=\dd u$ is simply the Lebesgue measure on the scale $u$, which is independent of $n$. In contrast, if say $\dd\nu$ is the counting measure on $\Z$, then $\dd\mu_n$ is the counting measure on the scaled lattice $\frac{1}{n}\Z$, hence varying with $n$.

Let us continue our discussion for a moment assuming that $\mu_n=\mu$ is independent of $n$. If the re-scaled kernel ${\msf K}_n(u,v)$ converges as $n\to\infty$ to a, say continuous, {\em limit kernel} $\msf K(u,v)$, uniformly for $u,v$ in compact subsets of $\K$, some probabilistic consequences for the point processes may be drawn. First, the limit kernel $\msf K$ then also defines a determinantal point process $\mathcal X$ on $\K$: this follows essentially from a criterion of Lenard described, e.g., in \cite{Soshnikov2000a}. {\rev Secondly, we have weak convergence of the point processes $\mathcal X_n$ to $\mathcal X$ as $n\to\infty$, meaning that \cite[Section 11.1]{DaleyVere-JonesVolII} the finite dimensional distributions of $\mcal X_n$ converge weakly (in the sense of measures) to the finite dimensional distributions of $\mcal X$.} For Hermitian positive-definite kernels, this follows from the fact that the associated operators converge in trace-norm, see e.g.\ \cite[Proposition 3.10]{ShiraiTakahashi}. In general, proving weak convergence of the point processes using convergence of the kernels requires more careful arguments, as we will explain in Remark \ref{remark:weakcvg}.

If $\mu_n$ depends on $n$, the same holds true if $\mu_n$ converges to some limiting measure $\mu$ in a sufficiently strong sense. In many situations, the integral operator $\mathbf K$ with kernel $\msf K$ will be an infinite rank projection on $L^2(\mu)$, hence not trace-class, but {\em locally trace-class}, i.e.\ its restriction to every bounded subset of $\K$ is trace-class. In such cases, random point configurations in $\mathcal X$ are a.s. infinite \cite{Soshnikov2000a} but locally finite, and we represent them as counting measures $\xi=\sum_i \delta_{u_i}$. 

In orthogonal polynomial ensembles with $\dd\nu(x)=\dd x$ and $w_n(x)=\ee^{-nV(x)}$ on $\K=\R$, such scaling limits are well understood. 

If $x^*$ is a (regular) point in the bulk of the spectrum, one has to take $\gamma=1$, and one finds the {\em sine kernel} $\msf K_{\sin}(u,v)\deff\frac{\sin \pi (u-v)}{\pi(u-v)}$ as limit kernel for a suitable choice of $c=c(x^*)>0$. If $x^*$ is a (regular) soft edge, one has to take $\gamma=2/3$ and an appropriate $c>0$, and the limit kernel is the {\em Airy kernel} 
$\msf K_{\rm Ai}(u,v)\deff\frac{\Ai(u)\Ai'(v)-\Ai(v)\Ai'(u)}{u-v}$.
Near points where $w_n$ is discontinuous, limit kernels related to other special functions, like the Bessel kernel or confluent hypergeometric kernel, arise, {\rev see for instance \cite{MorenoMartinezFinkelshteinSousa2010, BorodinOlshanski2001, TracyWidom1994Bessel, NagaoWadati1994, Forrester1993Bessel, KuijlaarsVanlessen2002HardEdge} for some of their early appearances}.
For choices of weight functions $w_n$ leading to singular bulk or edge points, limit kernels connected to Painlev\'e equations appear, see e.g. \cite{Duits14} for an overview. Similar scaling limits have also been found in a variety of biorthogonal ensembles beyond the orthogonal polynomial ensembles, though there is no completely understood classification of all different types of possible limit kernels. 

\subsection*{Weak convergence via probability generating functionals}\hfill \\
Instead of studying the convergence of point processes via their correlation functions, we prefer to study the {\it weak convergence} of point processes $\mcal X_n\stackrel{*}{\to} \mcal X$ in a more direct way.
{\rev We recall that weak convergence $\mcal X_n\stackrel{*}{\to} \mcal X$ of point processes means that the finite dimensional distributions of $\mcal X_n$ converge weakly (in the sense of measures) to the finite dimensional distributions of $\mcal X$.}  {In turn, such weak convergence is equivalent to point-wise convergence of the probability generating functionals \cite[Proposition 11.1.VIII]{DaleyVere-JonesVolII}: $\mathcal X_n$ converges weakly to $\mathcal X$ as $n\to\infty$ if and only if $\lim_{n\to\infty}\mathcal G_n[h]=\mathcal G[h]$ for every continuous function $h:\K\to [0,1]$ of bounded support with $\sup_{x\in \K} h(x)<1$,} where 
\begin{equation}\label{def:GnG}
\mathcal G_n[h]\deff \mathbb E_n \left[\prod_{k=1}^n(1-h(u_k))\right],\qquad \mathcal G[h]\deff \mathbb E \left[\prod_{k}(1-h(u_k))\right]
\end{equation}
are the probability generating functionals for the point processes $\mathcal X_n$ and $\mathcal X$ \cite[Definition 9.4.IV]{DaleyVere-JonesVolII}. In other words, weak convergence is the convergence of all (suitably regular) multiplicative observables of the particle system.

Naturally, another approach towards {\rev convergence results} consists in taking scaling limits directly on the level of probability generating functionals, 
instead of correlation kernels. This will enable us to prove convergence for a large class of weight functions without explicit expressions for correlation kernels at our disposal.
This approach is also rather natural as well from the physical perspective. In essence, we are approaching {\rev convergence} from observables of the model rather than from the correlation functions.

In our context of scaled determinantal point processes with kernels of locally trace-class integral operators ${\mathbf K}_n$, the condition $\lim_{n\to\infty}\mathcal G_n[h]=\mathcal G[h]$ translates into a limit condition for Fredholm determinants,
\begin{equation}
\lim_{n\to\infty}\det(\itid-\sqrt{h}{\mathbf K_n}\sqrt{h})_{L^2(\mu_n)}=\det(\itid-\sqrt{h}{\mathbf K}\sqrt{h})_{L^2(\mu)},
\end{equation}
which, in case $\mu_n=\mu$ is independent of $n$, is achieved if the operators $\sqrt{h}{\mathbf K_n}\sqrt{h}$ converge in trace norm to the operator $\sqrt{h}{\mathbf K}\sqrt{h}$ as $n\to\infty$.
Observe that $\det(\itid-\sqrt{h}{\mathbf K_n}\sqrt{h})_{L^2(\mu_n)}=\det(\itid-h{\mathbf K_n})_{L^2(\mu_n)}$ on the left-hand side, but the Fredholm determinant
$\det(\itid-h{\mathbf K})_{L^2(\mu)}$  is not defined unless if $h\mathbf K$ is trace-class on $L^2(\mu)$. 
Since $\mathbf K$ is locally trace-class and $h$ is assumed to have bounded support, $\sqrt{h}\mathbf K\sqrt{h}$ is trace-class, hence the Fredholm determinant $\det(\itid-\sqrt{h}{\mathbf K}\sqrt{h})_{L^2(\mu)}$ is well-defined. Therefore, we utilize the {\em symmetrized} operator $\sqrt{h}{\mathbf K_n}\sqrt{h}$ also on the left.

\subsection*{A class of deformations of biorthogonal ensembles}\hfill 

Let $(\mathcal X_n)_{n\in\mathbb N}$ be a biorthogonal ensemble of the form \eqref{def:BiOEscaled}, with correlation kernel ${\msf K}_n$ and probability generating functional $\mathcal G_n$. From now on, we always have in mind situations where $\dd\mu_n(u)$ is the reference measure on a microscopic scale, i.e.\ like in \eqref{def:wtP}, the distribution is induced by a scaling of a distribution of the form \eqref{def:biOPE}, so we view $\mcal X_n$ as the already appropriately scaled  biorthogonal ensemble, with corresponding scaled kernel $\msf K_n$ as above.

Given a measurable function $\sigma_n:\K \to [0,1]$ such that $\mathcal G_n[\sigma_n]\neq 0$, we define a deformed biorthogonal ensemble $\mathcal X_n^{\sigma_n}$ with probability distribution
\begin{equation}\label{def:deform}
\frac{1}{\msf Z_n\mathcal G_n[\sigma_n]}\det\left[F_k(u_j)\right]_{j,k=1}^{n}\det\left[{G_k(u_j)}\right]_{j,k=1}^{n} \prod_{j=1}^n (1-\sigma_n(u_j))W_n(u_j){\dd \mu_n(u_j)},
\end{equation}
and with associated probability generating functional 
\begin{multline*}
\mathcal G_n^{\sigma_n}[h]\deff\frac{1}{\msf Z_n\mathcal G_n[\sigma_n]}\int_{\K^n}\det\left[F_k(u_j)\right]_{j,k=1}^{n}
\det\left[{G_k(u_j)}\right]_{j,k=1}^{n} \\
\times
\prod_{j=1}^n (1-h(u_j))(1-
\sigma_n(u_j)) W_n(u_j){\dd \mu_n(u_j)}.
\end{multline*}
%
Any sequence $(\sigma_n)$ thus induces a deformation 
\begin{equation}\label{def:Wndeformed}
W_n^{\sigma_n}(u)\deff (1-\sigma_n(u))W_n(u),
\end{equation} 
of the weight function on microscopic scales; before rescaling, 
in \eqref{def:biOPE}, the deformation corresponds to the deformed weight function
\begin{equation}\label{def:deformedweight}
w_n^{\sigma_n}(x)=\left(1-\sigma_n\left(cn^\gamma(x-x^*)\right)\right)w_n(x).
\end{equation}

 Naturally, one is asked to understand to what extent the convergence of the original ensemble towards a certain limiting point process results in convergence of the deformed ensembles as well. The construction of the deformed ensemble \eqref{def:deform} relies on the explicit form of the joint distribution of biorthogonal ensembles, but the biorthogonality structure is not carried through when taking scaling limits. Hence, before studying convergence of such deformed ensembles, the first question is whether one can define similar deformations $\mathcal X^{\sigma}$ of the limit point processes $\mathcal X$ of the undeformed biorthogonal ensembles. To construct such limiting deformed ensembles, we rely on a probabilistic interpretation of the deformed biorthogonal ensemble $\mathcal{X}_n^{\sigma_n}$ developed in \cite{ClaeysGlesner2021}, as we describe next.

\subsection*{Conditional ensembles and universality.}\hfill

We start with a random point configuration $\xi=\sum_j \delta_{u_j}$ in a ground point process $\mcal X$ on the space $\K$ with reference measure $\mu$, and a sufficiently regular function $\sigma:\K\to [0,1]$ which we view as the symbol for a deformation. 
For our purposes, this ground process $\mcal X$ will be either a  biorthogonal ensemble of the form \eqref{def:BiOEscaled}, or a determinantal point process defined by a limit kernel like the sine or Airy kernel. We associate to each $u_j$ independently a random Bernoulli mark $m_j$, which is equal to $0$ with probability $1-\sigma(u_j)$ and equal to $1$ with complementary probability $\sigma(u_j)$. We then condition this marked point process on the event that all points have mark $m_j=0$, and we denote this conditional point process by $\mcal X^\sigma$. 
We may view the marked process as the observation of particles in an experiment, when receiving a mark $1$ means that the particle was effectively observed. The resulting conditional ensemble $\mcal X^\sigma$ is thus the process obtained under the condition that none of the particles were observed in the experiment. {Within such interpretation, the understanding of $\mcal X^\sigma$ provides insights on the lost data.}

It was proved in \cite{ClaeysGlesner2021} that the result of applying this procedure of marking (with a function $\sigma=\sigma_n$) and conditioning
on a biorthogonal ensemble $\mathcal X_n$ of the form \eqref{def:wtP} is precisely the deformed biorthogonal ensemble \eqref{def:deform}.
But the procedure is more general and applies to general determinantal point processes $\mathcal X$.
The resulting conditional process $\mathcal X^\sigma$ is defined whenever $\mathcal G[\sigma]\neq 0$, and it is still a determinantal point process, which we can see as the natural $n\to\infty$ generalization of the deformation \eqref{def:deform}.
Moreover, suppose that $\mathcal X$ has correlation kernel $\msf K$, and that the associated integral operator $\mathbf K$ is a locally trace-class operator on $L^2(\mu)$.
Then, it was proved in \cite{ClaeysGlesner2021} that
\begin{equation}\label{def:Ksigma}
\mathbf K^\sigma\deff \sqrt{1-\sigma}\mathbf K(\itid-\sigma\mathbf K)^{-1}\sqrt{1-\sigma}
\end{equation}
is a locally trace-class operator on $L^2(\mu)$, and that the integral kernel $\msf K^\sigma$ of this operator is the correlation kernel of the particles in the deformed determinantal point process $\mathcal X^\sigma$. The associated probability generating functional $\mathcal G^\sigma$ is thus given by
\begin{equation}\label{def:Gsigma}
\mathcal G^\sigma[h]=\sum_{k=0}^\infty\frac{(-1)^k}{k!}\int_{\K^k}\det\left(\msf K^\sigma(u_j,u_\ell)\right)_{j,\ell=1}^k \prod_{j=1}^k h(u_j)\dd \mu(u_j)=\det(\itid-\sqrt h\mathbf K^\sigma  \sqrt h)_{L^2(\mu)}.
\end{equation}
If $\sigma=0$, the deformed point process $\mathcal X^\sigma$ is of course equal to the undeformed $\mathcal X$.

{\rev If we start with $\mathcal X_n$ and assume $\mathcal X_n$ converges to a limit determinantal point process $\mathcal X$ as $n\to\infty$}, then what can we say about the large $n$ limit of $\mathcal X_n^{\sigma_n}$? From the perspective of universality, a natural rephrasal of this question would be: how far can we deform $\mcal X_n$ to a new process $\mcal X_n^{\sigma_n}$ while still preserving the same limiting process $\mcal X$? 

As our results will show, if we deform a biorthogonal ensemble on microscopic scales, for instance by taking $\sigma_n$ independent of $n$ in \eqref{def:Wndeformed}, then we lose universality of the limit process: under appropriate but mild conditions, we will actually prove that the deformed process $\mcal X_n^{\sigma_n}$ converges to the deformation $\mcal X^\sigma$ of the original limit process $\mcal X$ driven by the symbol $\sigma=\lim_{n\to\infty} \sigma_n$. 
This is still a natural phenomenon: when we deform the model on microscopic scales, we modify the limit point process, but the new limit point process is simply the corresponding deformation of the original limit point process.
From another perspective, our results drastically enlarge the class of possible limit kernels arising in biorthogonal ensembles, including in particular the classical sine, Airy, and Bessel point processes, by adding to it all $\sigma$-deformations of existing limit point processes. When the deformation takes place on sub-microscopic scales, we will show that the limit point process is not deformed, see Remark \ref{remark:submicroscopic}.
We will prove our results for a wide variety of biorthogonal ensembles under rather mild conditions, without assuming much regularity of the deformation symbol $\sigma_n$, and they will yield {\rev asymptotic results} of the deformed ensembles basically as consequences of the known universality results of the corresponding undeformed ensembles.

{\rev 

\begin{remark}
Let us also mention that the kernels $\msf K^\sigma$ are intimately connected to the theory of {\em integrable kernels} \cite{IIKS}.
These are kernels which can be written in the form
\[
\msf K(u,v)=\frac{\sum_{j=1}^kg_j(u)h_j(v)}{u-v},\qquad \text{with}\quad  \sum_{j=1}^kg_j(u)h_j(u)=0,
\]
like the sine and Airy kernels, as well as all other known limit kernels of orthogonal polynomial ensembles (with $k=2$).
If the kernel $\msf K$ is integrable, then the deformed kernel $\msf K^\sigma$ is also integrable (with the same value of $k$). Moreover, $\msf K^\sigma$ can then be characterized in terms of a $k\times k$ matrix-valued Riemann-Hilbert problem, which makes them amenable to asymptotic analysis, and which allows to study associated integrable differential equations, see \cite{ClaeysGlesner2021}. For the deformations of the sine, Airy, and Bessel kernels, the underlying differential equations are well understood \cite{CafassoClaeysRuzza2021, GhosalSilva22, ClaeysTarricone24, Ruzza24}, as well as the asymptotic behavior in certain regimes.
\end{remark}
}

\subsection*{Convergence theorem for non-varying measures}\hfill

The general setting of our main results starts with a sequence of $n$-point biorthogonal ensembles $\mathcal X_n$, $n\in\mathbb N$, of the form \eqref{def:BiOEscaled} with a reference measure $\mu_n$, kernels $\msf K_n$, and a sequence of deformations $\mathcal X_n^{\sigma_n}$ of the form \eqref{def:deform}--\eqref{def:Wndeformed}, induced by a sequence of functions $\sigma_n$, $n\in\mathbb N$. Our main result will ask for convergence of the original point process $\mcal X_n$ to a limiting process $\mcal X$ in a suitable sense, in order to conclude convergence of the deformed ensemble $\mcal X_n^{\sigma_n}$ towards a deformation $\mcal X^{\sigma}$ of the limiting process $\mcal X$. 

We will need suitable assumptions on $\mu_n$, $\sigma_n$ and $\msf K_n$, and we first consider the case of a non-varying measure $\mu_n=\mu$. Loosely speaking, in this case we need (1) that the deformed ensembles are well-defined, (2) that $\msf K_n$ and $\sigma_n$ converge point-wise to a limit kernel $\msf K$ and a function $\sigma$, and (3) a rough domination bound for the kernels $\msf K_n$.

\begin{assumption}\label{assumption:sigma} Let $\mu_n=\mu$ be a Radon measure on $\K$ which is independent of $n$. The kernels $\msf K_n:\K^2\to\mathbb C$ {\rev as in \eqref{def:rescaled kernel}} and the functions $\sigma_n:\K\to [0,1]$ are such that the following conditions hold.
\begin{enumerate}
\item 
The deformed ensemble \eqref{def:deform} is well-defined, i.e.\ $\mathcal G_n[\sigma_n]>0$ for every $n\in\mathbb N$.
\item
There exist a determinantal point process $\mathcal X$ on $\K$ with kernel $\msf K$ and probability generating functional $\mathcal G$, and a function $\sigma:\K\to[0,1]$ for which $\mathcal G[\sigma]\neq 0$, such that for $\mu$-a.e. $u,v\in\K$, we have the {\it point-wise} limits
\[
\lim_{n\to\infty}{\msf K}_n(u,v)={\msf K}(u,v)\quad\mbox{and} \lim_{n\to\infty}\sigma_n(u)=\sigma(u),\quad \text{for }\mu\text{-a.e. } u,v\in \K .
\]
\item For every bounded $F\subset\K$, there exist $\Phi,\Psi\in L^2(\mu)$, possibly depending on $F$, such that
$$
\sqrt{\itid_F(u)+\sigma_n(u)}|{\msf K}_n(u,v)|\sqrt{\itid_F(v)+\sigma_n(v)}\leq \Phi(u)\Psi(v),\quad \text{for }\mu\text{-a.e. } u,v\in \K  \text{ and every }n,
$$
{\rev where $\itid_F$ is the indicator function of the set $F$.}
\end{enumerate}\end{assumption}

The following is a general convergence result for deformed biorthogonal ensembles.
\begin{theorem}\label{thm:biOPE}
Let  $(\mathcal X_n)_{n\in\mathbb N}$ be a sequence of biorthogonal ensembles of the {\rev form \eqref{def:BiOEscaled}} with corresponding sequence of correlation kernels $({\msf K}_n)_{n\in \N}$. Let $(\sigma_n)_{n\in\mathbb N}$ be a sequence of Borel measurable functions $\sigma_n:\mb K \to[0,1]$. If Assumption \ref{assumption:sigma} holds,
then the limit
\begin{equation}\label{eq:scalingGdeformed}
\lim_{n\to\infty}\mathcal G_n^{\sigma_n}[h]=\mathcal G^\sigma[h],
\end{equation}
holds true for any continuous function $h:\mb K \to [0,+\infty)$ with bounded support.
In particular, the deformed biorthogonal ensemble $\mathcal X_n^{\sigma_n}$ converges weakly as $n\to\infty$ to the deformed limit point process $\mathcal X^{\sigma}$.
\end{theorem}

Theorem~\ref{thm:biOPE} is proven in Section~\ref{sec:proofB}.
\begin{remark}\label{remark:weakcvg}
The simplest case of the above result occurs when $\sigma_n\equiv 0$, which means that we do not deform the biorthogonal ensembles $\mathcal X_n$. 
In this case, Assumption \ref{assumption:sigma} only means dominated convergence of the kernels $\msf K_n$ to $\msf K$.
Even if this is not the most interesting situation, it is still worth to mention that Theorem \ref{thm:biOPE} then gives a practical sufficient condition to deduce weak convergence of the point processes $\mathcal X_n$ to $\mathcal X$ from suitable convergence of the correlation kernels $\msf K_n$ to $\msf K$.
Note moreover that, if the kernels $\msf K_n(u,v)$ are continuous and the convergence of $\msf K_n$ to $\msf K$ is uniform on compacts, then Assumption \ref{assumption:sigma} is always valid for $\sigma_n\equiv 0$. 
\end{remark}

\begin{remark}\label{remark:submicroscopic}
Another simplified version of Theorem \ref{thm:biOPE} occurs when the sequence $\sigma_n$ is not identically $0$, but the limit function $\sigma$ is identically zero. This happens for instance if we set $\sigma_n(u)=h(n^\epsilon u)$ for some $\epsilon>0$ and for some function $h$ decaying to $0$ at infinity, which means that the deformation of the biorthogonal ensemble takes place on sub-microscopic scales.
Then, if Assumption \ref{assumption:sigma} holds, Theorem \ref{thm:biOPE} states that the deformed biorthogonal ensembles $\mathcal X_n^{\sigma_n}$ converge weakly to the undeformed limit point process $\mathcal X$: as announced informally before, limit point processes are insensitive to sub-microscopic deformations.
\end{remark}

We view Theorem~\ref{thm:biOPE} as a perturbative tool, close to the spirit of universality: once the convergence of the original kernel $\msf K_n$ is established, we can extend the convergence of the ground point processes $(\mcal X_n)$ to their deformed versions $(\mcal X_n^{\sigma_n})$ after a routine check
of Assumption \ref{assumption:sigma}. Unlike the traditional approach through an asymptotic analysis of the correlation kernels, our approach does not require a separate analysis of the deformed ensembles. In this direction, in Section~\ref{sec:applic} we will apply Theorem~\ref{thm:biOPE} to several concrete situations, extending known convergence results of the original sequences $(\mcal X_n)$ to their deformed versions $(\mcal X_n^{\sigma_n})$, for large families of deformations $\sigma_n$. Notably, thanks to Theorem~\ref{thm:biOPE} we will provide a characterization of the correlation kernel of various different conditional-thinned deformations of the sine and Airy point processes in terms of integrable systems previously studied in the literature in depth, see Remark~\ref{remark:intsysdefkernels} later on. 


\subsection*{Convergence theorem for varying measures}\hfill \\
There is an analogue to Theorem~\ref{thm:biOPE} for varying measures $\mu_n$, which in particular will allow us to consider biorthogonal ensembles over lattices. In addition to the previous assumptions on $\msf K_n$ and $\sigma_n$, we need to assume a suitable type of convergence of the varying measures $\mu_n$ to a measure $\mu$. In the assumption below and in what follows, for functions $f,g$ we use the notation $(f\otimes g)(x,y)=f(x)g(y)$, and denote by $\mu\otimes\lambda$ the product measure on $\mathbb K^2$ generated by two measures $\mu$ and $\lambda$.

\begin{assumption}\label{assumption:sigmavarying} 
The sequence of Radon measures $(\mu_n)$, the kernels $\msf K_n:\K^2\to\mathbb C$ {\rev as in \eqref{def:rescaled kernel}}, and the functions $\sigma_n:\K\to [0,1]$ are such that the following hold.
\begin{enumerate}[(1)]
\item The deformed ensemble \eqref{def:deform} is well-defined, i.e.\ $\mathcal G_n[\sigma_n]>0$ for every $n\in\mathbb N$.
\item There exist a determinantal point process $\mathcal X$ on $\K$ with {\em continuous} kernel $\msf K$ and associated probability generating functional $\mathcal G$, and a {\it continuous} function $\sigma:\K\to [0,1]$, such that for any compact $F\subset \K $, the limits
$$
\lim_{n\to\infty}\|(\itid_F\otimes \itid_F)({\msf K}_n-{\msf K})\|_{L^\infty(\mu_n\otimes \mu_n)}=0 \qquad\mbox{and} \qquad \lim_{n\to\infty}\|\itid_F(\sqrt{\sigma_n}-\sqrt{\sigma})\|_{L^\infty(\mu_n)}=0
$$
are valid.

\item There is a Radon measure $\mu$ over $\K $ for which for every bounded $F\subset\K $, there exist {\it continuous} functions $\Phi,\Psi \in L^2(\mu)\cap L^2(\mu_n)$, possibly depending on $F$ but independent of $n$, such that
$$
\sqrt{\itid_F(u)+\sigma_n(u)}|{\msf K}_n(u,v)|\sqrt{\itid_F(v)+\sigma_n(v)}\leq \Phi(u)\Psi(v),
$$
for $\mu_n$-a.e. $u,v\in\K $ and for every $n$. Furthermore, the norms 
$$
\|\Phi\|_{L^2(\mu_n)}, \|\Phi\|_{L^\infty(\mu_n)}, \|\Psi\|_{L^2(\mu_n)} \quad \text{and}\quad \|\Psi\|_{L^\infty(\mu_n)}
$$ 
remain bounded as $n\to \infty$, and for these same functions the weak convergence of measures
$$
\Phi(x)\Psi(x) \dd\mu_n(x) \stackrel{*}{\to}\Phi(x)\Psi(x)\dd\mu(x) \text{ as }N\to \infty
$$
takes place.
\end{enumerate}
\end{assumption}
In Assumption~\ref{assumption:sigmavarying}, the measures $\mu_n$ and $\mu$ need not be finite, so it is not suitable to talk about their weak convergence directly. However, since $\Phi,\Psi\in L^2(\mu)\cap L^2(\mu_n)$, each of the measures $\Phi\Psi\dd\mu$ and $\Phi\Psi\dd\mu_n$ is finite, and the weak convergence from Assumption~\ref{assumption:sigmavarying} (3) makes sense.

We emphasize that Assumption~\ref{assumption:sigmavarying} (2)--(3) requires an understanding of the kernel $\msf K_n(u,v)$ for $u,v$ in the varying set $\supp\mu_n$. In concrete discrete models, the $n$-dependent kernel $\msf K_n$ is naturally defined only on $\supp\mu_n$, but not on the whole space, and its limit $\msf K$ admits a continuous extension to the whole space $\K $. This point will be clarified in the concrete example of discrete Coulomb gases discussed in Section~\ref{sec:discreteCoulomb}, and explains why the pointwise convergence and bounds in Assumption~\ref{assumption:sigma} (2)--(3) are naturally replaced with Assumption~\ref{assumption:sigmavarying} (2).

The next result is the previously announced analogue of Theorem~\ref{thm:biOPE} for varying measures.

\begin{theorem}\label{thm:biOPE-varying}
Let  $(\mathcal X_n)_{n\in\mathbb N}$ be a sequence of biorthogonal ensembles of the form \eqref{def:wtP} with correlation kernels ${\msf K}_n$ with a varying reference measure $\dd\mu_n$. Let $(\sigma_n)_{n\in\mathbb N}$ be a sequence of Borel measurable functions $\sigma_n:\K \to[0,1]$. If Assumption \ref{assumption:sigmavarying} holds, then we have the limit
\begin{equation}\label{eq:scalingGdeformed2}
\lim_{n\to\infty}\mathcal G_n^{\sigma_n}[h]=\mathcal G^\sigma[h]
\end{equation}
for any continuous function $h:\K \to [0,+\infty)$ with bounded support.
In particular, the deformed biorthogonal ensemble $\mathcal X_n^{\sigma_n}$ converges weakly as $n\to\infty$ to the deformed limit point process $\mathcal X^{\sigma}$.
\end{theorem}

We emphasize that Assumption~\ref{assumption:sigmavarying} asks for the continuity solely of the limiting functions $\msf K$ and $\sigma$, but not of $\msf K_n,\msf \sigma_n$. Assumption~\ref{assumption:sigmavarying} involves conditions on $\msf K_n,\sigma_n$ that must be verified on $\supp\mu_n$ but not on the whole space. The limiting point processes $\mcal X$ and $\mcal X^\sigma$ are determinantal with respect to the kernel $\msf K$ as acting on $L^2(\mu)$. With these observations in mind, functions $\sigma_n,\msf K_n$ need to be defined only $\mu_n$-a.e., and the continuity of $\msf K,\sigma,\Psi,\Phi$ on the whole space $\K $ can be replaced by the continuity of these functions on a neighborhood of $\supp\mu$, and Theorem~\ref{thm:biOPE-varying} still remains valid under such conditions.

Theorem~\ref{thm:biOPE-varying} is proven in Section~\ref{sec:proofvarying}.
It is particularly useful when dealing with discrete biorthogonal ensembles, for which the underlying state spaces are scaled with $n$ to a continuum limit. We discuss some applications in this direction in Section~\ref{sec:discreteCoulomb}.

\section{Applications of the general results}\label{sec:applic}

We now illustrate how our main theorems apply in several models of interest.

\subsection{Bulk deformations of orthogonal polynomial ensembles}\hfill

Consider orthogonal polynomial ensembles \eqref{def:OPE} with varying measure $w_n(x)=\ee^{-nV(x)}$.
The correlation kernel $\msf k_n$ can then be taken of the form \eqref{def:kernel}, with $\phi_k=\psi_k$ being the normalized orthogonal polynomial of degree $k-1$ with respect to the weight $w_n(x)$ on $\mathbb R$.
%
%

%
%
Let us assume that 
\begin{equation}\label{eq:conditionsOPintro}
V \text{ is of class $C^2$ and } \lim_{x\to \pm\infty} V(x)/\log(1+x^2)=+\infty.
\end{equation}
These rather weak regularity conditions ensure that we work under the framework developed by Lubinsky \cite{Lubinsky09, Lubinsky08} and further extended together with Levin \cite{LevinLubinsky11, LevinLubinsky08} for the study of universality. We now collect some known results for these orthogonal polynomial ensembles that we will need.

A first result is concerned with the global limit density of the ensemble: there exists an {\em equilibrium density} $\kappa_V:\mathbb R\to [0,+\infty)$ supported on a finite number of compact intervals, for which
\begin{equation}\label{eq:eqdens}
\lim_{n\to\infty}\frac{1}{n}\msf k_n(x,x)=\kappa_V(x),
\end{equation}
uniformly for $x$ in compact subsets of the interior of the support of $\kappa_V$. The associated measure $\kappa_V(x)\dd x$ is known as the equilibrium measure, as it is the solution of a minimization problem from logarithmic potential theory \cite[Chapter 6]{deift_book}. The convergence \eqref{eq:eqdens} has been proved in the literature under various different settings and forms. If $V$ is real analytic the limit \eqref{eq:eqdens} holds uniformly for $x\in \R$ as inferred from \cite{Deiftetal1999}, whereas under the weaker regularity we are assuming this result is uniform on compact subsets of the interior of $\supp \kappa_V$ as can be seen from \cite{Totik00}. 

A second result is concerned with local correlations between eigenvalues. Let us fix a reference point $x^*\in\mathbb R$ for which $\kappa_V(x^*)>0$; such a point is referred to as a {\em bulk point}. Consider the scaling {\rev $x=x^*+\frac{u}{\kappa_V(x^*)n}$} and the corresponding scaled kernel
\begin{equation}
\msf K_n(u,v)\deff \frac{1}{\kappa_V(x^*)n}\msf k_n\left(x^*+\frac{u}{\kappa_V(x^*)n},x^*+\frac{v}{\kappa_V(x^*)n}\right),
\end{equation}
i.e., the kernel defined in \eqref{def:rescaled kernel} with $c=\kappa_V(x^*)$ and $\gamma=1$.
Then, the following sine kernel limit is a classical example of bulk universality:
\begin{equation}\label{eq:sinelimit}
\lim_{n\to\infty}\msf K_n(u,v)=\msf K_{\sin}(u,v)\deff \frac{\sin(\pi(u-v))}{\pi(u-v)},
\end{equation}
point-wise for $u,v\in\mathbb R$. {\rev This limit was proven by Pastur and Shcherbina in \cite{pastur_shcherbina_universality} under weak conditions on $V$, and shortly later established using Riemann-Hilbert methods (when $V$ is real analytic) by Deift and collaborators using Riemann-Hilbert methods \cite{deift_kriecherbauer_mclaughlin, deift_book, Deiftetal1999, Deiftetal99b}}. The general result under {\rev the weaker condition \eqref{eq:conditionsVhmm} below was proved by Levin and Lubinsky in \cite{LevinLubinsky08}.} The limit kernel $\msf K_{\sin}$ is the sine kernel, and it defines a determinantal point process, the sine point process, which we denote as $\mathcal X_{\sin}$.

Finally, we need a rough uniform bound for the kernel on the diagonal, or in other words a bound on the one-point correlation function: there exists a constant $C>0$ for which the inequality
\begin{equation}\label{eq:globalestknbulk}
\left|\frac{1}{\kappa_V(x^*)n}\msf k_n\left(x^*+\frac{u}{\kappa_V(x^*)n}, x^*+\frac{u}{\kappa_V(x^*)n}\right)\right| \leq C
\end{equation}
holds true for any $u\in \R$ and any $n$. For $u$ in compacts, this follows from \cite[Equation~(4.13)]{LevinLubinsky08}. It extends to the full real line with the help of the so-called restricted range inequalities, for instance with an application of \cite[Chapter~III, Theorem~2.1]{Saff_book} - {\rev and following their notation} - for the weight $w=\ee^{-V/2}$ and polynomial $P_{2n}=\sum_{k=0}^{n-1}\phi_{k}^2$, for which $P_{2n}(x)w(x)^{2n}=\msf k_n(x,x)$, compare with \eqref{def:kernel}.

Let us now introduce a suitable deformation of this ensemble. 
We take an integrable function $f:\mathbb R\to [0,+\infty)$ 
which is such that $f(x)=O(|x|^{-\eta})$ as $x\to \pm\infty$, for some $\eta>1$.
Define $\sigma_n$ by the equation
\begin{equation}\label{def:sigmaOPE}
1-\sigma_{n}(u)=\ee^{-f(n^tu)},\qquad t\geq 0.
\end{equation}
The deformation induced by $\sigma_n$ lives on microscopic scales if $t=0$ and on sub-microscopic scales if $t>0$.
Since $1-e^{-y}\leq y$ for all $y\in\mathbb R$, we can bound
\begin{align}
\sigma_n(u) & =1-\ee^{-f(n^tu)}\leq f(n^t u)\nonumber \\ 
& \leq C\itid_{[-M,M]}(u)+\frac{C}{|u|^{\eta}} \itid_{\mathbb R\setminus[-M,M]}(u)\revdeff H(u), \label{eq:boundsigmansine}
\end{align}
for some constants $C,M>0$. Observe that $H:\mathbb R\to\mathbb R$ is an integrable function independent of $n$.

We proceed by proving that Assumption \ref{assumption:sigma} is valid in this setting.
Condition (1) is easily verified, since $1-\sigma_n(u)>0$ for a.e. $u\in\mathbb R$. 
To see that the point-wise limits in part (2) are valid, we recall \eqref{eq:sinelimit}, and we observe that
\begin{equation}\label{eq:sigmalimitsine}
\lim_{n\to\infty}\sigma_n(u)=\sigma(u)\deff \begin{cases}
1-\ee^{-f(u)},&t=0,\ u\in\mathbb R,\\
0,&t>0,\ u\in\mathbb R\setminus\{0\}.
\end{cases}
\end{equation}
To prove that part (3) of Assumption \ref{assumption:sigma} is valid, we recall \eqref{eq:globalestknbulk}. Using that $\msf K_n$ is the Christoffel-Darboux kernel for orthogonal polynomials, we apply the Cauchy-Schwarz inequality to bound $\msf K_n(u,v)$ in terms of its diagonal values, and obtain
\begin{equation}\label{eq:CS01}
|\msf K_n(u,v)|\leq \sqrt{\msf K_n(u,u)\msf K_n(v,v)}\leq C.
\end{equation}
By construction of $\sigma_n$, we have for any compact $F\subset\mathbb R$ that
\begin{equation}
\sqrt{\itid_F(u)+\sigma_n(u)}|{\msf K}_n(u,v)|\sqrt{\itid_F(v)+\sigma_n(v)}\leq C\sqrt{\itid_F(u)+H(u)}\sqrt{\itid_F(v)+H(v)},\qquad t\geq 0.
\end{equation}
Since $\Phi(u)=\Psi(u)=\sqrt{C(\itid_F(u)+H(u))}$ is square-integrable, part (3) of Assumption \ref{assumption:sigma} is valid for $t\geq 0$.

Applying Theorem \ref{thm:biOPE}, we now obtain the following result.

\begin{corollary}\label{cor:sine}
For any compactly supported continuous function $h:\R\to [0,+\infty)$, we have the limits
\[
\lim_{n\to\infty}\mathcal G_n^{\sigma_n}[h]=\begin{cases}
\mathcal G_{\sin}^{\sigma}[h],&t=0,\\
\mathcal G_{\sin}[h],&t>0.\end{cases}
\]
 where $\mathcal G_{\sin}$ and $\mathcal G_{\sin}^\sigma$ are the probability generating functionals of the sine point process and its $\sigma$-deformation,
\[
\mathcal G_{\sin}[h]=\det(\itid- \sqrt h\mathbf K_{\sin}\sqrt h),\qquad \mathcal G_{\sin}^{\sigma}[h]=\det(\itid- \sqrt h\mathbf K_{\sin}^\sigma\sqrt h),
\]
with 
\[
\mathbf K_{\sin}^\sigma=\sqrt{1-\sigma}\mathbf K_{\sin}(\itid-\sigma\mathbf K_{\sin})^{-1}\sqrt{1-\sigma}.
\]
In particular, the deformed orthogonal polynomial ensemble $\mathcal X_n^{\sigma_n}$ converges weakly as $n\to\infty$ to the deformed sine point process $\mathcal X_{\sin}^{\sigma}$ if $t=0$, and to the sine point process $\mathcal X_{\sin}$ if $t>0$.
\end{corollary}

Let us discuss some special cases that appeared previously in the literature. 

\begin{enumerate}[(1)]
\item If $\sigma(u)=\itid_{(-s,s)}$ with $s>0$, the deformation of the orthogonal polynomial ensemble $\mathcal X_n$ consists of restricting the domain of its particles to $\mathbb R\setminus[-s,s]$. Accordingly, the deformed sine point process $\mathcal X_{\sin}^{\itid_{(-s,s)}}$ is the sine point process conditioned on the event that a random point configuration contains no points in $(-s,s)$. The Riemann-Hilbert problem characterizing the kernel $\msf K_{\sin}^{\itid_{(-s,s)}}$ and a connection to the Painlevé V equation were analyzed in the groundbreaking works \cite{deift_its_zhou_riemann_hilbert_random_matrices_integrable_statistical_mechanics} and \cite{JimboMiwaMoriSato}, respectively.
\item If $\sigma(u)=\gamma \itid_{(-s,s)}$ with $s>0$ and $\gamma\in(0,1)$, the deformation of the orthogonal polynomial ensemble $\mathcal X_n$ and of the sine point process are obtained by conditioning on the event that a random thinned point configuration contains no points in $(-s,s)$. Here, the random thinned point configuration is the point configuration obtained after removing each point independently with probability $1-\gamma$.
For $\gamma\in (0,1]$, the limit kernel $\msf K_{\sin}^{\gamma \itid_{(-s,s)}}$ was expressed in terms of the solution of a Riemann-Hilbert problem in \cite{Bothner2015}, and its Fredholm determinant was associated to the Painlev\'e V equation in \cite{MehtaPV}, see also \cite{DysonPV}.
\item If $\sigma$ is of Schwartz class, the limit kernel $\msf K_{\sin}^{\sigma}$ can be expressed in terms of the solution of a Riemann-Hilbert problem
connected to an integro-differential generalization of the Painlev\'e V equation, and is connected to an integrable PDE introduced in \cite{IIKS} which is an isospectral deformation of the Zakharov-Shabat system. {\rev This observation is computed explicitly in \cite{CandidoAlvesChouteau2025}, based on the results from \cite{ClaeysTarricone24}}.
\end{enumerate}

Case (3) corresponds to $\sigma$ as in \eqref{eq:sigmalimitsine}, with $t=0$. Case (2) corresponds to $t=0$ and the choice $f(u)=-\log(1-\gamma \itid_{(-s,s)}(u))$ with $\gamma\in (0,1)$, which is a function of compact support and as such satisfying the decay conditions for $f$.

In turn, Case (1) corresponds to $t=0$ and formally to $f(u)=-\log(1-\itid_{(-s,s)(u)})$, so that $f(u)=+\infty$ on $(-s,s)$. In particular, such $f$ is not integrable. Nevertheless, this case can still be recovered in a similar way as in Corollary~\ref{cor:sine}, simply by setting $\sigma_n(u)=\sigma(u)=H(u)=\itid_{(-s,s)}(u)$ directly, instead of \eqref{def:sigmaOPE}.

\subsection{Edge deformations of orthogonal polynomial ensembles}\hfill

Here, we consider the same class of orthogonal polynomial ensembles as in the previous subsection, namely we take \eqref{def:OPE} with $w_n(x)=\ee^{-nV(x)}$, but now with $V$ satisfying the stronger  condition
\begin{equation}\label{eq:conditionsVhmm}
V \text{ is real analytic, strictly convex, and } \lim_{x\to \pm\infty} V(x)/\log(1+x^2)=+\infty. 
\end{equation}
The real analyticity is needed to apply the results from
\cite{deift_kriecherbauer_mclaughlin, deift_book, Deiftetal1999, Deiftetal99b}; similar results have been obtained under weaker conditions in 
 \cite{LevinLubinsky08} and
\cite{McLaughlinMiller08}, but require more technical assumptions, and for the sake of brevity we choose not to pursue in this direction.
The strict convexity ensures that the support of the equilibrium density $\kappa_V$ is a single interval, which we denote as $[x^-,x^+]$.
Instead of focusing on bulk points, we now fix one of the edge points as a reference point, and for the sake of simplicity, we will henceforth consider the right edge point $x^+$. 
Thanks to the strict convexity of $V$, it is known that $x^+$ is a regular soft edge, meaning that there exists $C>0$ (depending on $V$ and $x^+$) such that
\begin{equation}\label{def:softedge}\kappa_V(x)\sim C\sqrt{x^+-x},\qquad x\to (x^+)_-. \end{equation}
In this situation, the proper scaling \eqref{def:rescaled kernel} is with $\gamma=2/3$ and $c=\left(\frac{\pi}{2}C\right)^{2/3}$ with $C$ as in \eqref{def:softedge}. In other words, we consider the re-scaled orthogonal polynomial ensemble $\mathcal X_n$
with kernel
\[\msf K_n(u,v)=\frac{1}{cn^{2/3}}\msf k_n\left(x^*+\frac{u}{cn^{2/3}}, x^*+\frac{v}{cn^{2/3}}\right).\]
We then have the edge scaling limit
\begin{equation}\label{eq:edgelimit}
\lim_{n\to\infty} \msf K_n(u,v)=\msf K_{\rm Airy}(u,v)\deff 
{
\rev\frac{\ai(u)\ai'(v)-\ai(v)\ai'(u)}{u-v},
}
\end{equation}
uniformly for $u,v$ in compacts of the real line,
which was proved in \cite{Deiftetal1999, Deiftetal99b}.
The limit kernel $\msf K_{\rm Airy}$ is the Airy kernel, and it induces the Airy point process which we denote as $\mcal X_{\rm Airy}$. This limit process is universal, in the sense that it does not depend on $V$, and in fact it appears in a wide range of models outside the realm of random matrix theory, {\rev see e.g. \cite{Konig05,Johansson06}}. 

Like in the bulk, we need also a rough bound for the one-point function $\msf K_n(u,u)$, which must be uniform in $u\in\mathbb R$. It follows also from the asymptotics for $\msf K_n(u,u)$ obtained in \cite{Deiftetal1999, Deiftetal99b} that there exists a constant $C'>0$ such that
\begin{equation}
\label{eq:uniformboundedge}
\msf K_n(u,u)\leq {\rev C'}(\sqrt{|u|+1})\mbox{ for $u<0$,}\qquad \msf K_n(u,u)\leq {\rev C'}\ee^{-V(u)}\mbox{ for } u\geq 0.
\end{equation}
This was not stated explicitly in this form in \cite{Deiftetal1999, Deiftetal99b}, but such estimates can be deduced from the asymptotic analysis of the Riemann-Hilbert problem in \cite{Deiftetal1999, Deiftetal99b}.
More precisely, for $|u|\leq \delta n^{2/3}$ with $\delta>0$ sufficiently small, the estimate follows from the construction of the local Airy parametrix together with the asymptotic behavior of the Airy function and Airy kernel, while for $|u|>\delta n^{2/3}$, it follows from the construction of the outer parametrix.

Similarly to the bulk case, we now introduce a suitable deformation of $\mathcal X_n$. As before, we take a locally integrable function $f:\mathbb R\to [0,+\infty)$ and define $\sigma_n$ by \eqref{def:sigmaOPE}.
The conditions we need to impose on $f$ are different from those in the bulk case, because they are dictated by the (different from the bulk) uniform bound \eqref{eq:uniformboundedge}.
Here, we need $f$ to decay sufficiently fast as $x\to -\infty$ such that $f(x)=O(|x|^{-3/2-\epsilon})$ for some $\epsilon>0$, and we also assume that 
$$
\lim_{x\to +\infty}f(x)=L \quad \text{ exists, possibly with } \quad L=+\infty.
$$
These conditions imply that there exists $M>0$ such that for $n$ sufficiently large and for all $t\geq 0$,
\[
\sigma_n(u)=1-\ee^{-f(n^t u)}\leq \itid_{[-1,+\infty)}(u) + M|u|^{-3/2-\epsilon} \itid_{(-\infty,-1)}(u).
\]
Consequently,
\begin{equation}\label{eq:fundboundAiry}
\sigma_n(u)\msf K_n(u,u)\leq 
C \ee^{-V(u)}\itid_{[-1,+\infty)}(u) + CM|u|^{-1-\epsilon}\itid_{(-\infty,-1)}(u) \revdeff H(u),
\end{equation}
and $H$ is integrable {\rev and independent of n}.

We now prove that Assumption \ref{assumption:sigma} is valid.
Condition (1) is again easily verified, since $1-\sigma_n$ is non-negative and does not vanish identically. 
For part (2), we recall \eqref{eq:edgelimit}, and we observe that
\begin{equation}\label{def:sigmaedge1}
\lim_{n\to\infty}\sigma_n(u)=\sigma(u)\deff 
1-\ee^{-f(u)}\qquad\mbox{ if $t=0$, $u\in\mathbb R$},
\end{equation}
while
\begin{equation}\label{def:sigmaedge2}
\lim_{n\to\infty}\sigma_n(u)=\sigma(u)\deff\begin{cases}
0,&u<0\\
1-\ee^{-f(0)},&u=0\\
1-\ee^{-L},&u>0
\end{cases},\qquad \mbox{if $t>0$.}
\end{equation}
It remains to prove part (3) of Assumption \ref{assumption:sigma}. By the Cauchy-Schwarz inequality (recall for instance the first inequality in \eqref{eq:CS01}), we have for any compact $F\subset\mathbb R$ that
\begin{multline*}
\sqrt{\itid_F(u)+\sigma_n(u)}|{\msf K}_n(u,v)|\sqrt{\itid_F(v)+\sigma_n(v)} \\ 
\leq \sqrt{(\itid_F(u)+\sigma_n(u))\msf K_n(u,u)}\sqrt{(\itid_F(v)+\sigma_n(v))\msf K_n(v,v)}.
\end{multline*}
Moreover, by \eqref{eq:uniformboundedge}, we have
\[(\itid_F(u)+\sigma_n(u))\msf K_n(u,u)\leq 2C\ee^{-V(u)},\qquad u\geq 0,\]
and
\[(\itid_F(u)+\sigma_n(u))\msf K_n(u,u)\leq C\itid_F(u)(1+\sqrt{|u|})+H(u),
\qquad u< 0.\]
We set \[\Phi(u)=\Psi(u)\deff 
\begin{cases}\sqrt{2C\ee^{-V(u)}},&u\geq 0\\
\sqrt{C\itid_F(u)(1+\sqrt{|u|})+H(u)},&u<0,
\end{cases}\]
such that $\Phi,\Psi\in L^2(\mathbb R)$, and part (3) of Assumption \ref{assumption:sigma} is valid.

Applying Theorem \ref{thm:biOPE}, we now obtain the following result.

\begin{corollary}\label{cor:Airy}
For any compactly supported continuous function $h:\R\to [0,+\infty)$, the limit
\[
\lim_{n\to\infty}\mathcal G_n^{\sigma_n}[h]=\mathcal G_{\rm Airy}^{\sigma}[h],
\]
is valid, {\rev where $\mathcal G_{\rm Airy}^\sigma$ is the probability generating functional of the $\sigma$-deformation of the Airy point process,
\[
\mathcal G_{\rm Airy}^{\sigma}[h]=\det(\itid- \sqrt h\mathbf K_{\rm Airy}^\sigma\sqrt h),
\]
}
with
\[
\mathbf K_{\rm Airy}^\sigma=\sqrt{1-\sigma}\mathbf K_{\rm Airy}(\itid-\sigma\mathbf K_{\rm Airy})^{-1}\sqrt{1-\sigma},
\]
and $\sigma$ given by \eqref{def:sigmaedge1}--\eqref{def:sigmaedge2}.
In particular, the deformed orthogonal polynomial ensemble $\mathcal X_n^{\sigma_n}$ converges weakly as $n\to\infty$ to the deformed Airy point process $\mathcal X_{\rm Airy}^{\sigma}$.
\end{corollary}

Let us again consider some concrete examples.

\begin{enumerate}[(1)]
\item For $\sigma(u)=\itid_{(s,+\infty)}$, the deformed {\rev orthogonal polynomial} ensemble is the original one restricted to the domain $(-\infty,s)$. In other words, it consists of introducing a hard edge in the vicinity of the soft edge, which corresponds to $u=0$. This leads to a hard-to-soft-edge transition, and the limit kernels $\msf K_{\rm Airy}^{\itid_{(s,+\infty)}}$ in this transition have been studied in \cite{claeys_kuijlaars08} and are connected to the Painlev\'e II equation; it was however not noticed there that the limit point process is the Airy point process conditioned on having no particles bigger than $s$.
\item For $\sigma(u)=\gamma \itid_{(s,+\infty)}(u)$ with $s\in\mathbb R$ and $\gamma\in (0,1)$, the limit kernels $\msf K_{\rm Airy}^{\sigma}$ and the associated Riemann-Hilbert problems have been investigated in detail in \cite{XuZhao, BogatskiyClaeysIts16}. These kernels are connected to the Hastings-McLeod and Ablowitz-Segur solutions of the Painlev\'e II equation.
\item For the choice $\sigma(u)=1/(1+\ee^{-tu-s})$ with $t>0$ and $s\in \R$, the kernel $\msf K_{\rm Airy}^{\sigma}$ was obtained in \cite{GhosalSilva22} as the limit at the soft edge scaling for a class of microscopic deformations of  orthogonal polynomial ensembles of the form \eqref{def:OPE}. These deformations were in particular assumed to be real analytic in a neighborhood of the real axis. Here, we obtain a more general class of limit kernels parametrized by the function $\sigma$, and we obtain it under weak regularity assumptions on the deformation. It was also proved in  \cite{GhosalSilva22} that the limit kernel $\msf K_{\rm Airy}^{\sigma}$ can be expressed in terms of an integro-differential generalization of the second Painlev\'e equation, by combining results from \cite{CafassoClaeysRuzza2021} and \cite{ClaeysGlesner2021}.
\end{enumerate}

In a similar vein as for the sine process, functions $\sigma$ for which Case (3) applies correspond to the general ones in \eqref{def:sigmaedge1}--\eqref{def:sigmaedge2} with $t=0$, and Case (2) correspond to $t=0$ and $f(u)=-\log(1-\gamma\itid_{(s,+\infty)}(u))$ with $\gamma\in (0,1)$. In these two cases Corollary~\ref{cor:Airy} applies directly. 

Case (1) corresponds to $t=0$, and to apply Theorem~\ref{thm:biOPE} we modify the arguments outlined before Corollary~\ref{cor:Airy} as follows. With $m=\min(\inf F,s)$ the bound $\itid_F(u)+\sigma(u)\leq 2\times \itid_{(m,+\infty)}(u)$ applies, and the function $H$ in \eqref{eq:fundboundAiry} must be replaced by $C\times \itid_{(m,+\infty)}(u)\ee^{-V(u)}$, which is integrable. The arguments in the sequel of \eqref{eq:fundboundAiry} can then be applied.

Along the lines of (3), a combination of \cite{CafassoClaeysRuzza2021} and \cite{ClaeysGlesner2021} similar to the one explored in \cite{GhosalSilva22}, together with Corollary~\ref{cor:Airy}, yields the characterization of the kernel $\msf K_{\rm Airy}^\sigma$ in terms of an integro-differential Painlevé II equation, for a large class of $\sigma:\R\to [0,1]$. Gap probabilities for a family of kernels which also contain the kernels $\msf K_{\rm Airy}^\sigma$ for some choices of $\sigma$ have been recently studied by Kimura and Navand \cite{KimuraNavand2024}. 

\begin{remark}\label{remark:intsysdefkernels}
As mentioned, in all the cases (1)--(3) mentioned after Corollary~\ref{cor:sine}, and also in the cases (1)--(2) mentioned just above, the deformed kernel $\msf K^\sigma$ was characterized in terms of a solution to a {Riemann-Hilbert problem. The mentioned works established the connection of the kernels with integrable differential equations (more specifically Painlev\'e equations)
 and considered asymptotic questions related to them. Curiously, the probabilistic characterization of such kernels as the correlation kernels of the conditional thinned versions of the sine and Airy point processes was not noticed before, despite the fact that it can help to obtain more insight in asymptotic questions, for instance to construct $g$-functions needed in asymptotic Riemann-Hilbert analysis.}

Even though we focused here on discussing deformations near bulk and soft edge points, our general result Theorem \ref{thm:biOPE} can be used to obtain similar results near hard edges, near discontinuities of the weight functions, or near other types of singular points where limiting kernels other than the sine and Airy kernels arise \cite{GhosalSilva22,ClaeysGlesner2021}.    
\end{remark}

%

\subsection{Deformations of discrete Coulomb gases}\label{sec:discreteCoulomb}\hfill 

The orthogonal polynomial ensembles \eqref{def:OPE} can be interpreted as continuous log-gases. They admit natural discrete analogues that we now discuss. For positive integers $n$ and $N$ with $n<N$, consider discrete measures of the form
\begin{equation}\label{eq:discreteCoulomb}
P_{n,N}(x_1,\cdots, x_n)\prod_{i=1}^n \dd \nu_N(x_j)\deff \frac{1}{\msf Z_{n,N}}\prod_{1\leq i<j\leq n}(x_j-x_i)^2\prod_{j=1}^n w_N(x_j)\dd \nu_N(x_j),
\end{equation}
viewed as a distribution of $n$ particles $x_j$ on the discrete set 
\begin{equation}\label{deff:equallyspacedlattice}
\Lambda=\Lambda_N\deff \left\{x^{(N)}_{j}: j=0,\ldots, N-1\right\}\subset \R.
\end{equation}
We assume $\Lambda_N\subset [0,1]$ for simplicity, the considerations that follow would still go through replacing $[0,1]$ by any other compact subset of the real axis. The points $x_j^{(N)}$ are usually referred to as the {\it nodes} for the discrete ensemble. The set of nodes $\Lambda_N$ varies with $N$, and $\dd\nu_N$ is the $N$-dependent counting measure on $\Lambda_N$. The partition function $Z_{n,N}$ is taken such that \eqref{eq:discreteCoulomb} is a probability measure. The ensemble \eqref{eq:discreteCoulomb} is a biorthogonal ensemble as well, in fact a discrete orthogonal polynomial ensemble. The functions $\phi_j=\psi_j=p_{j-1}^{(N)}$ in \eqref{eq:biorth} are the orthonormal polynomials with respect to the discrete varying measure $w_N(x)\dd\nu_N(x)$ on $\Lambda_N$, namely satisfying
\begin{equation}\label{eq:discreteOPs}
\int_{\Lambda_N} p^{(N)}_j(x)p^{(N)}_k(x) w_N(x)\dd\nu_N(x)=\sum_{x\in \Lambda_N} p^{(N)}_j(x)p^{(N)}_k(x) w_N(x)=\delta_{j,k},\quad j,k=0,\hdots, n,
\end{equation}
where $\delta_{j,k}$ is the Kronecker delta, and conditions on the weight $w_N$ will be placed in a moment. Particular instances of such orthogonal polynomial ensembles arise with the rescaled Krawtchouk and Hahn weights, in connection with random domino tilings of the Aztec diamond, and random lozenge tilings of the hexagon, respectively {\rev \cite{Johansson2005RWHahn, Johansson2001, Johansson2002nonintersecting}}. Since their introduction in the seminal work \cite{Johansson2001}, discrete particle systems of the form \eqref{eq:discreteCoulomb} on various different subsets of the integers have been profoundly explored in the literature, in particular in connection with growth models and tilings as the ones mentioned, among many other ones. We refer the reader to \cite{Baik2007a, DasDimitrov2022} for an account of the overarching connection of \eqref{eq:discreteCoulomb} with other models.

For us, we consider \eqref{eq:discreteCoulomb} under the framework studied in \cite{Baik2007a}, which we now briefly describe. In summary, we work under the following assumptions.
\begin{assumption}\label{assumption:dCoulomb}
The discrete Coulomb gas ensemble \eqref{eq:discreteCoulomb} satisfies the following conditions.
\begin{enumerate}[(i)]
    \item The set of nodes $\Lambda_N$ is defined by a quantization rule \eqref{eq:quantization} with respect to a density function $\rho$.
    \item The orthogonality weight is of the form $w_N(x)= \ee^{-NV_N(x)}$ for a suitable function $V_N$.
    \item The constrained equilibrium measure of the system satisfies a regularity condition {\rev (in the sense discussed around \eqref{eq:energyfunctional} below)}.
   \end{enumerate}    
\end{assumption}

Under these assumptions, we are interested in the limit where $n\to +\infty$ with $n=\beta N$ and $\beta\in (0,1)$ fixed.
We now detail what parts (i)--(iii) of Assumption \ref{assumption:dCoulomb} mean. In what follows, for a compactly supported finite measure $\lambda$ on $\R$, its logarithmic potential is
$$
U^\lambda(z)\deff \int \log \frac{1}{|x-z|}\dd\lambda(x),\quad z\in \R,
$$
which is a well-defined function assuming values in $(-\infty,+\infty]$. 

We assume that the set of nodes $\Lambda_N$ is defined through a density function: there exists a function $\rho$ which is analytic in a complex neighborhood of $[0,1]$, $\rho(x)>0$ on $[0,1]$, $\int_0^1\rho(x)\dd x=1$, and the nodes are defined through the relation
\begin{equation}\label{eq:quantization}
\int_0^{x^{(N)}_{j}}\rho(x) \dd x=\frac{2j+1}{2N},\quad j=0,\hdots, N-1,
\end{equation}
which is usually referred to as a quantization rule.
We view $\rho$ as a density on $[0,1]$, with its corresponding logarithmic potential $U^\rho \deff U^{\rho\itid_{[0,1]}\dd x}$.
The weight $w_N$ is assumed to be of the form
\begin{equation}\label{eq:discreteCoulombweight}
w_N(x)=\ee^{-NV_N(x)}, \quad \text{with}\quad V_N(x)=V(x)-U^{\rho}(x)+\frac{\eta(x)}{N},
\end{equation}
and where $V$ is independent of $N$, real analytic in a complex neighborhood $G$ of $(0,1)$, and $\eta$ is allowed to depend on $N$ but in such a way that
\begin{equation}\label{eq:boundetadC}
C_\eta\deff \limsup_{N\to\infty}\sup_{z\in G} |\eta(z)|<\infty.
\end{equation}
The function $V$ may be viewed as the confining potential for the discrete Coulomb gas. The term $U^\rho$ in \eqref{eq:discreteCoulombweight} could have been incorporated into $V$, but it has a slightly different meaning: it arises upon taking the discrete-to-continuous limit of the discrete set $\Lambda_N$ to the interval $[0,1]$. Furthermore, for technical reasons it is important that $V$ admits an analytic continuation to a full complex neighborhood of $[0,1]$, which is not necessarily the case for $U^\rho$.

The situation of equally spaced nodes is recovered with $\rho\equiv 1$ and $x_j^{(N)}=\frac{2j-1}{2N}$, for which $U^\rho(x)=-x\log x-(1-x)\log(1-x)+1$. The particular cases of (rescaled) Hahn and Krawtchouk polynomials correspond to equally spaced nodes and the choice $V(x)=-(a+x)\log(a+x)-(b-x)\log (b-x)+c x +d$, where $a,b,c,d$ are appropriate parameters.

A central role is played by the constrained equilibrium measure $\mu^V$, which is the unique Borel probability measure on $[0,1]$ that minimizes
\begin{equation}\label{eq:energyfunctional}
\iint \frac{1}{|x-y|} \dd\mu(x)\dd\mu(y)+\frac{1}{\beta}\int (V(x)-U^\rho(x))\dd \mu(x),
\end{equation}
over all Borel probability measures $\mu$ on $[0,1]$ satisfying the constraint $\mu(B)\leq \frac{1}{\beta}\int_B \rho(x)\dd x$ for every Borel set $B\subset [0,1]$. 

The constraint implies that the equilibrium measure $\mu^V$ is absolutely continuous with respect to the Lebesgue measure, say $\dd\mu^V(x)=\kappa_V(x)\dd x$, with ${\rev 0\leq \kappa_V(x)\leq \rho(x)/\beta}$. Moreover, in our situation, the interval $[0,1]$ is divided into a finite number of bands, voids, and saturated regions: a void is an interval on which ${\rev \kappa_V(x)=0}$, a band is an interval on which $0<\kappa_V(x)<\rho(x)/\beta$, and a saturated region is an interval on which ${\rev \kappa_V(x)=\rho(x)/\beta }$. The density $\kappa_V$ is real analytic on each of the bands.  {\rev We assume that $\kappa_V(x)\dd x$ is a regular equilibrium measure, in the following sense.
\begin{enumerate}
\item[(i)] Both $\kappa_V(x)$ and $\rho(x)/\beta-\kappa_V(x)$ have a square-root behavior at each point where they vanish.
\item[(ii)] The two endpoints $0$ and $1$ do not belong to the closure of a band.
\end{enumerate}

Condition (i) is a generic condition on $V$, in the sense that if it is not satisfied by a specific $V$, then it is satisfied for any sufficiently small perturbation of $V$; we refer the reader to \cite[Section~2.1]{Baik2007a} for a detailed discussion. Condition (ii) implies that $\kappa_V(0),\kappa_V(1)$ are either $0$ or $\rho(0)/\beta,\rho(1)/\beta$. }
We impose all these technical conditions in order to be able to use convergence results and bounds for the kernel $\msf k_n$ from established sources in the literature. The computation of the constrained equilibrium measures for the Hahn and Krawtchouk orthogonal polynomial ensembles are thoroughly discussed in \cite[Section~2.4]{Baik2007a} and \cite{DragnevSaff00}, respectively.

Under the conditions just discussed, it is known that $\kappa_V$ describes the one-point function of the ensemble \eqref{eq:discreteCoulomb}. Concretely, if $J$ is any closed interval where $0<\kappa_V(x)<\frac{\rho(x)}{\beta}$,
 then the limit
\begin{equation}\label{eq:densitydCoulomb}
\lim_{n\to\infty} \msf k_n(x,x) = \kappa_V(x),
\end{equation}
holds true as $n\to \infty$ with $n=\beta N$ as before, uniformly for $x\in J$. As stated, this is \cite[Lemma~7.12]{Baik2007a}. Variants of this result also appear in \cite{DragnevSaff97, Feral08, johansson_2000}.

We will now apply Theorem~\ref{thm:biOPE-varying} to the discrete Coulomb gas under Assumption~\ref{assumption:dCoulomb}. For that, we need to verify  Assumption~\ref{assumption:sigmavarying}, which we do as follows.
\begin{enumerate}[(I)]
    \item Determine the scaling of the kernel \eqref{def:rescaled kernel} and the corresponding varying measures $\mu_n$ that, together, determine a point process $\mcal X_n$ with state space $\Omega_N$. 
    \item Identify a limiting measure $\mu$ and the state space $\Omega$ where it acts.
    \item Determine the limiting kernel $\msf K$ acting on $L^2(\mu)$, with the corresponding limiting point process $\mcal X$ on the state space $\Omega$.
    \item Determine a suitable class of varying symbols $\sigma_n$ and their limit $\sigma$.
    \item Verify that, in this construction, Assumption~\ref{assumption:sigmavarying} (1)--(3) holds.
\end{enumerate}

We consider a {\it bulk point}, that is, a point $x^*$ for which $0<\kappa_V(x^*)<\rho(x^*)/\beta$. 
Scale the set of nodes $\Lambda_N$ to $\Omega_N$,
$$
\Omega_N\deff \frac{\kappa_V(x^*)n}{\beta} (\Lambda_N-x^*),
$$
and scale $\msf k_n$ to $\msf K_n$ as in \eqref{def:rescaled kernel}, here taking the form
\begin{equation}\label{eq:corrkerneldiscrete}
\msf K_n(u,v)\deff \msf k_n\left( x^*+\frac{\beta u}{\kappa_V(x^*)n}, x^*+\frac{\beta v}{\kappa_V(x^*)n} \right), \quad u,v\in \Omega_N.
\end{equation}
Notice that we are not adding the usual factor $\frac{1}{n \kappa_V(x^*)}$ in front of the scaled kernel. This is so because we prefer not to scale the counting measure of $\Omega_N$ later on, and this choice is also consistent with \eqref{eq:densitydCoulomb}.

For later reference, we emphasize that $\msf k_n$ is defined through the discrete orthonormal polynomials \eqref{eq:discreteOPs} with varying weight $w_N$, namely
$$
\msf k_n(x,y)\deff \sqrt{w_N(x)w_N(y)}\sum_{k=0}^{n-1} p_k^{(N)}(x)p_k^{(N)}(y),\quad x,y\in \Lambda_N.
$$
Notice that $\Omega_N$ is naturally a finite set of $\K=\R$, and $\msf K_n$ is the correlation kernel of a point process on $\Omega_N$, with the reference measure $\dd\mu_n$ being the counting measure on $\Omega_N$, hence varying with $n=\beta N$. The kernel \eqref{eq:corrkerneldiscrete}, is initially defined only for $u,v\in \Omega_N$. We extend $\msf K_n$ to other values $u,v\in \R$ by setting it to zero for $u,v\in \R^2\setminus (\Omega_N\times \Omega_N)$. Observe that verifying Assumption~\ref{assumption:sigmavarying} (2)--(3) requires only information about the kernel on $\supp\mu_n=\Omega_N$, so such {\rev an} extension is only needed for completeness but does not play any substantial role in what follows. This concludes step (I).

To identify the state space of the limiting point process, let us label
$$
\Omega_N=\{ \cdots <q^{(N)}_{-1}<0\leq q_0^{(N)}< q_2^{(N)}<\cdots  \}=\{q_j^{(N)}\}_{j\in J}, \quad J=J_N\subset \Z.
$$
That is, $\Omega_N$ is a collection of $N$ sites $q_j^{(N)}$ indexed over $J\subset \Z$, and $q_0^{(N)}$ is the smallest non-negative of such sites.

Condition \eqref{eq:quantization} written in terms of the $q_j^{(N)}$'s implies that these nodes must satisfy
\begin{equation}\label{eq:quantizationscaled}
\int_{q_{j-1}^{(N)}}^{q_j^{(N)}}\rho\left(x^*+\frac{\beta u}{\kappa_V(x^*)n}\right)\dd u=\frac{\kappa_V(x^*)}{\beta}\frac{n}{N}=\kappa_V(x^*).
\end{equation}
The density $\rho$ is continuous and strictly positive on $[0,1]$, say with $m\leq \rho(x)\leq M$ for every $x\in [0,1]$. This last identity implies in particular that
\begin{equation}\label{eq:bounddiscrnodes}
\frac{\kappa_V(x^*)}{M} \leq q_j^{(N)}-q_{j-1}^{(N)}\leq \frac{\kappa_V(x^*)}{m},
\end{equation}
for every $j$. From the fact that $q_0^{(N)}$ is the smallest non-negative site we obtain the inequality $|q_0^{(N)}|\leq |q_0^{(N)}-q_{-1}^{(N)}|$, and therefore the sequence $(q_0^{(N)})_N$ remains bounded as $N\to \infty$. From now on we assume for definiteness that $q_0^{(N)}\to 0$, otherwise all the considerations that will follow still hold true after a suitable shift.

The bound \eqref{eq:bounddiscrnodes} and the convergence $q_0^{(N)}$ implies that for each $j$ fixed, the sequence $(q_j^{(N)})_N$ remains bounded as $N\to \infty$. Thanks to \eqref{eq:quantizationscaled}, we must thus have
\begin{equation}\label{eq:limitnodes}
q_j^{(N)}\to \frac{\kappa_V(x^*)}{\rho(x^*)} j \quad \text{ as }N\to \infty \text{ while } j \text{ is kept fixed}.
\end{equation}
Thus, we identified that the set of nodes $\Omega_N$ is converging to the limiting set $\Omega\deff \frac{\kappa_V(x^*)}{\rho(x^*)} \Z$. More precisely, recalling that $\dd\mu_n$ is the counting measure for $\Omega_N$, and denoting by $\dd\mu$ the counting measure for $\Omega$, the limit \eqref{eq:limitnodes} implies that the sequence $(\mu_n)$ converges vaguely to $\mu$, that is, for any continuous function $h$ with compact support we have $\int h \,\dd\mu_n\to \int h\, \dd\mu$ as $N\to \infty$. This concludes step (II).

For later, we also need the following rough concentration property: there exists $C_0>0$ such that for any interval $[-L,L]$ and any $n$ sufficiently large,
\begin{equation}\label{eq:unifboundcontingemeas}
\mu_n([-L,L])\leq C_0 L.
\end{equation}
Such a concentration is a consequence of \eqref{eq:quantizationscaled}, \eqref{eq:bounddiscrnodes} and \eqref{eq:limitnodes}, and carries through to $\mu([-L,L])=\lim_{n\to\infty} \mu_n([-L,L])$.

As we will see in a moment, the kernel $\msf K_n$ converges to the {\it discrete sine kernel}, defined on $\Omega$ by
$$
\msf K(u,v)=\msf K_{\dsin}(u,v)\deff \frac{\beta \kappa_V(x^*)}{\rho(x^*)}\frac{\sin(\pi(u-v))}{\pi(u-v)},\quad u\neq v,\quad \msf K_{\dsin}(u,u)\deff \frac{\beta \kappa_V(x^*)}{\rho(x^*)}
$$
In other words, the expression for $\msf K_\dsin$ is a multiple of the sine kernel in \eqref{eq:sinelimit}, but we reserve the notation $\msf K_\dsin$ to remind the reader that this kernel acts as an operator $\bm K_\dsin$ on the space $\ell^2(\Omega)$ with the discrete counting measure $\mu$. Nevertheless, observe that $\msf K_\dsin$ is continuous on $\R^2=\K^2$, a condition which is necessary in Assumption~\ref{assumption:sigmavarying}. The kernel $\msf K_\dsin$ defines a point process $\mcal X_\dsin$ on the discrete set $\Omega$. Step (III) is therefore concluded.

We consider functions $\sigma_n(u)$ of the form
\begin{equation}\label{eq:sigmadsine}
\sigma_n(u)=\sigma(u)=1-\ee^{-f(u)},\quad u\in \R,
\end{equation}
where $f:\R\to [0,+\infty)$ is continuous and compactly supported. In particular, $\sigma_n$ does not depend on $n$, and is interpreted as already living on the microscopic scale $u$. In principle, we could replace the assumption of compact support of $f$ by a mild growth condition on $f$ (compare with \eqref{def:sigmaOPE}), but that would require some mild global bounds on the kernel $\msf K_n$ which we could not find in the literature, and obtaining them would go beyond the scope of this paper. This concludes step (IV).

Finally, we now verify (V).

We start by analyzing the kernel $\msf K_n$. First of all, given a compact $F\subset \R$, \cite[Lemma~7.13]{Baik2007a} says that
\begin{equation}\label{eq:dsinconv}
\sup_{u,v\in F\cap \Omega_N} | \msf K_n(u,v)- \msf K_\dsin(u,v) |\leq \frac{C}{n},
\end{equation}
where $C=C_F(x^*)$ is a constant depending on $F$, the bulk point $x^*$ and $\beta\in (0,1)$, and the inequality is valid as $n\to \infty$ with $n=\beta N$. This statement yields the first limit in Assumption~\ref{assumption:sigmavarying} (2). The second limit in Assumption~\ref{assumption:sigmavarying} (2) is trivial, as in our case of consideration $\sigma_n=\sigma$ is independent of $n$.

Finally, we verify {Assumption~\ref{assumption:sigmavarying} (3)}. For that, we bound the kernel $\msf K_n$ and the factor $\sqrt{\itid_F+\sigma_n}$ separately. For the kernel, the convergence \eqref{eq:dsinconv} and the boundedness of $\msf K_\dsin$ give that for any $L>0$, a bound of the form
\begin{equation}\label{eq:boundkerneldOP0}
|\itid_{[-L,L]}(u)\msf K_n(u,v)\itid_{[-L,L]}(v)|\leq C \itid_{[-L,L]}(u)\itid_{[-L,L]}(v),\quad u,v\in \Omega_N
\end{equation}
holds true. 

Next, take $L>0$ sufficiently large such that $F\cup \supp f\subset [-L,L]$, making also sure that $\pm L\notin \Omega$. From the bound \eqref{eq:boundkerneldOP0},
$$
|\sqrt{\itid_F(u)+\sigma_n(u)}\msf K_n(u,v)\sqrt{\itid_F(v)+\sigma_n(v)}|\leq 2C\, \itid_{[-L,L]}(u)\itid_{[-L,L]}(v) \quad u,v\in \Omega_N.
$$
Hence, setting $\Psi(u)=\Phi(u)=2C\, \itid_{[-L,L]}(u)$, the bound in Assumption~\ref{assumption:sigmavarying} (3) holds true. Thanks to \eqref{eq:unifboundcontingemeas}, the $L^2(\mu_n)$ and $L^\infty(\mu_n)$ norms of both $\Phi$ and $\Psi$ remain bounded as $n\to \infty$. Because $\pm L\notin \Omega=\supp\mu$ and $(\mu_n)$ converges vaguely to $\mu$, the weak convergence $\Phi\Psi \dd\mu_n\stackrel{*}{\to}\Phi\Psi \dd\mu$ also holds true. Step (V) is finally completed, and using Theorem~\ref{thm:biOPE-varying} we concluded the following.

\begin{corollary}\label{cor:dCoulomb}
Consider the discrete Coulomb gas ensemble \eqref{eq:discreteCoulomb} under Assumption~\ref{assumption:dCoulomb}. Scale it around a bulk point $x^*$ as in \eqref{eq:corrkerneldiscrete} and let $\sigma$ be as in \eqref{eq:sigmadsine}.

For any compactly supported continuous function $h:\R\to [0,+\infty)$, the limit
\[
\lim_{n\to\infty}\mathcal G_n^{\sigma_n}[h]=\mathcal G_{\dsin}^{\sigma}[h],
\]
is valid, where $\mathcal G_{\dsin}$ and $\mathcal G_{\dsin}^\sigma$ are the probability generating functionals of the discrete sine point process on $\Omega=\frac{\kappa_V(x^*)}{\rho(x^*)}\Z$ and its $\sigma$-deformation,
\[
\mathcal G_{\dsin}[h]\deff \det(\itid- \sqrt h\mathbf K_{\dsin}\sqrt h),\qquad \mathcal G_{\dsin}^{\sigma}[h]\deff \det(\itid- \sqrt h\mathbf K_{\dsin}^\sigma\sqrt h),
\]
with
\[
\mathbf K_{\dsin}^\sigma\deff \sqrt{1-\sigma}\mathbf K_{\dsin}(\itid-\sigma\mathbf K_{\dsin})^{-1}\sqrt{1-\sigma}.
\]
In particular, the deformed discrete Coulomb gas $\mathcal X_n^{\sigma_n}$ converges weakly as $n\to\infty$ to the deformed discrete sine point process $\mathcal X_{\dsin}^{\sigma}$.
\end{corollary}

\section{Proof of Theorem \ref{thm:biOPE}}\label{sec:proofB}

In this section, we consider biorthogonal ensembles $\mathcal X_n$ as in \eqref{def:BiOEscaled} with kernels $\msf K_n$ as in \eqref{def:rescaled kernel} and probability generating functionals $\mathcal G_n$, acting on an underlying space $L^2(\mu)$ {\rev over the field $\C$}.
Given a sequence of functions $(\sigma_n)$, we consider the deformed biorthogonal ensemble $\mathcal X_n^{\sigma_n}$ from \eqref{def:deform}, and we assume that $\msf K_n$ and $\sigma_n$ are such that Assumption \ref{assumption:sigma} holds true.
As before, we denote $\msf K_n^{\sigma_n}$ for the correlation kernel of $\mathcal X_n^{\sigma_n}$, and $\mathcal G_n^{\sigma_n}$ for its probability generating functional; when applying this functional to a function $h$, we will always assume that $h:\K \to [0,1)$ is continuous and compactly supported, with $\sup_{x\in\K }h(x)<1$. 
The next two auxiliary results are essentially special cases of more general results in \cite{ClaeysGlesner2021}, but we include a self-contained proof here in our settings, for the convenience of the reader.

\begin{lemma}\label{lemma:G}
If Assumption \ref{assumption:sigma} (1) holds, we have the identity
\[\mathcal G_n^{\sigma_n}[h]=\det\left(\itid-\sqrt h{\mathbf K}_n^{\sigma_n}\sqrt h\right),\quad\mbox{with}\quad {\mathbf K}_n^{\sigma_n}\deff \sqrt{1-\sigma_n}{\mathbf K}_n\left(\itid-\sigma_n{\mathbf K}_n\right)^{-1}\sqrt{1-\sigma_n}.\]
\end{lemma}
\begin{proof}
Recall from our general discussion on biorthogonal ensembles that the kernel $\msf K_n$ can be taken of the form
\[\msf K_n(u,v)=\sqrt{W_n(u)W_n(v)}\sum_{j=1}^n\Phi_j(u){\Psi_j(v)},\]
where $\Phi_1,\ldots, \Phi_n$ have the same linear span {\rev over $\C$, say $A_n$,} as $F_1,\ldots, F_n$, and $\Psi_1,\ldots, \Psi_n$ have the same linear span {\rev over $\C$, say $B_n$,} as $G_1,\ldots, G_n$, and they satisfy the biorthogonality relations
\[
\int_{\K } \Phi_j(u){\Psi_k(u)}W_n(u)\dd\mu(u)=\delta_{jk}.
\]
It is then straightforward to verify that the integral operator ${\mathbf K}_n$ with kernel $\msf K_n$ acting on $L^2(\mu)$ is the unique linear projection operator with range $\sqrt{W_n}A_n$ and kernel the orthogonal complement of $\sqrt{W_n}\,\overline{B_n}$ on $L^2(\mu)$, where $\overline{B_n}$ is the space of functions whose complex conjugate lies in $B_n$. 
Similarly, the integral operator corresponding to the correlation kernel $\msf K_n^{\sigma_n}$ of $\mathcal X_n^{\sigma_n}$ is the unique linear projection operator with range $\sqrt{(1-\sigma_n)W_n}A_n$ and kernel the orthogonal complement of $\sqrt{(1-\sigma_n)W_n}\,\overline{B_n}$.
We will now prove that this projection operator is given precisely by
\[
\mathbf K_n^{\sigma_n}\deff \sqrt{1-\sigma_n}{\mathbf K}_n\left(\itid-\sigma_n{\mathbf K}_n\right)^{-1}\sqrt{1-\sigma_n}.
\]
The operator $\bm K_n$ is finite rank, so $\sigma_n\bm K_n$ is trace-class. Recalling Assumption \ref{assumption:sigma} (1), 
\[
\mathcal G_n[\sigma_n]=\det\left(\itid-\sqrt{\sigma_n}\bm K_n\sqrt{\sigma_n}\right)=\det\left(\itid-{\sigma_n}\bm K_n\right)>0,
\] 
so that $\itid-{\sigma_n}\bm K_n$ is an invertible $L^2(\mu)$-operator.
%

To see that $\mathbf K_n^{\sigma_n}$ is indeed the desired projection, we observe first that it is of rank $n$, and that
\begin{align*}
\left(\mathbf K_n^{\sigma_n}\right)^2&=\sqrt{1-\sigma_n}{\mathbf K}_n\left(\itid-\sigma_n{\mathbf K}_n\right)^{-1}(1-\sigma_n){\mathbf K}_n\left(\itid-\sigma_n{\mathbf K}_n\right)^{-1}\sqrt{1-\sigma_n}
\\
&=\sqrt{1-\sigma_n}{\mathbf K}_n\left(\itid-\sigma_n{\mathbf K}_n\right)^{-1}(\itid-\sigma_n\mathbf K_n){\mathbf K}_n\left(\itid-\sigma_n{\mathbf K}_n\right)^{-1}\sqrt{1-\sigma_n}\\
&=\sqrt{1-\sigma_n}{\mathbf K}_n^2\left(\itid-\sigma_n{\mathbf K}_n\right)^{-1}\sqrt{1-\sigma_n}\\
&=\mathbf K_n^{\sigma_n},\end{align*}
so that it is indeed a rank $n$ projection. Furthermore, for $\Phi\in A_n$, we have 
$$
(\itid-\sigma_n\bm K_n)[\sqrt{W_n}\Phi]=\sqrt{W_n}\Phi-\sigma_n\bm K_n[\sqrt{W_n}\Phi]=(1-\sigma_n)\sqrt{W_n}\Phi
$$
because $\bm K_n$ is a projection onto $A_n$. Or, in other words, $(\itid-\sigma_n\bm K_n)^{-1}[(1-\sigma_n)\sqrt{W_n}\Phi]=\sqrt{W_n}\Phi$, so that
\begin{align*}\mathbf K_n^{\sigma_n}[\sqrt{(1-\sigma_n)W_n}\Phi]&=\sqrt{1-\sigma_n}{\mathbf K}_n\left(\itid-\sigma_n{\mathbf K}_n\right)^{-1}[(1-\sigma_n)\sqrt{W_n}\Phi]\\
&=\sqrt{1-\sigma_n}{\mathbf K}_n[\sqrt{W_n}\Phi]\\&=\sqrt{(1-\sigma_n)W_n}\Phi,\end{align*}
such that the range  
of $\mathbf K_n^{\sigma_n}$ is indeed
$\sqrt{(1-\sigma_n)W_n}A_n$.
For $F$ in the orthogonal complement of $\sqrt{(1-\sigma_n)W_n}\, \overline{B_n}$,
we have $\bm K_n[\sqrt{1-\sigma_n}F]=0$, so $(\itid-\sigma_n\bm K_n)[\sqrt{1-\sigma_n}F]=\sqrt{1-\sigma_n}F$, and
\begin{align*}
\mathbf K_n^{\sigma_n}[F]&=\sqrt{1-\sigma_n}{\mathbf K}_n\left(\itid-\sigma_n{\mathbf K}_n\right)^{-1}[\sqrt{1-\sigma_n}F]\\
&=\sqrt{1-\sigma_n}{\mathbf K}_n[\sqrt{1-\sigma_n}F]=0,\end{align*}
which concludes the proof.
\end{proof}

\begin{lemma}\label{lemma:GB}
If Assumption \ref{assumption:sigma} (1) holds, $\mathcal G_n^{\sigma_n}[h]$ is given by
$$
\mathcal G_n^{\sigma_n}[h]=\frac{\det\left(\itid-[\sigma_n+h-\sigma_n h]\mathbf K_n\right)}{\det\left(\itid-\sigma_n\mathbf K_n\right)}
$$
\end{lemma}
\begin{proof}
We recall from Lemma \ref{lemma:G} that the probability generating functional $\mathcal G_n^{\sigma_n}[h]$ is given by
\[\mathcal G_n^{\sigma_n}[h]=\det(\itid-\sqrt h\mathbf K_n^{\sigma_n}\sqrt h)=\det(\itid-h\mathbf K_n^{\sigma_n})
=
\det\left(\itid-h(1-\sigma_n)\mathbf K_n(\itid-\sigma_n\mathbf K_n)^{-1}\right).\]
If $\det\left(\itid-\sigma_n\mathbf K_n\right)\neq 0$, which is true by Assumption \ref{assumption:sigma} (1), we rewrite the latter as
\[\det\left(\left[\itid-\sigma_n\mathbf K_n-h(1-\sigma_n)\mathbf K_n\right](\itid-\sigma_n\mathbf K_n)^{-1}\right)=\frac{\det\left(\itid-[\sigma_n+h-\sigma_n h]\mathbf K_n\right)}{\det\left(\itid-\sigma_n\mathbf K_n\right)},\]
and the proof is complete.
\end{proof}

Our strategy now consists in proving that 
\begin{align}
&\label{eq:limittoprove1B}
\begin{multlined}[b]
\lim_{n\to\infty} \det\left(\itid-\sqrt{\sigma_n+h-\sigma_n h}\mathbf K_n\sqrt{\sigma_n+h-\sigma_n h}\right)= \\ 
\det\left(\itid-\sqrt{\sigma+h-\sigma h}\mathbf K\sqrt{\sigma+h-\sigma h}\right),
\end{multlined}
\\
&\label{eq:limittoprove2B}
\lim_{n\to\infty}\det\left(\itid-\sqrt{\sigma_n}\mathbf K_n\sqrt{\sigma_n}\right)=\det\left(\itid-\sqrt\sigma\mathbf K\sqrt\sigma\right).
\end{align}
This last determinant is non-zero by Assumption \ref{assumption:sigma} (2).
By Lemma \ref{lemma:GB}, \eqref{eq:limittoprove1B}--\eqref{eq:limittoprove2B} imply that
\begin{equation}
\lim_{n\to\infty}\mathcal G_n^{\sigma_n}[h]=\frac{\det\left(\itid-\sqrt{\sigma+h-\sigma h}\mathbf K\sqrt{\sigma+h-\sigma h}\right)}{\det\left(\itid-\sqrt\sigma\mathbf K\sqrt\sigma\right)}=\mathcal G^\sigma[h],
\end{equation}
where the last equality follows from  \cite[Theorem 2.4 (2)]{ClaeysGlesner2021}.
Thus, to complete the proof of Theorem \ref{thm:biOPE} under Assumption \ref{assumption:sigma}, it remains to prove \eqref{eq:limittoprove1B}--\eqref{eq:limittoprove2B}.

For any bounded Borel measurable function $\psi:\K \to[0,+\infty)$, we have that $\sqrt{\psi}\mathbf K_n\sqrt{\psi}$ is finite rank and therefore trace-class on $L^2(\mu)$, and
$$
\det\left(\itid-\sqrt{\psi}\mathbf K_n\sqrt{\psi}\right)=\sum_{k=0}^\infty\frac{(-1)^k}{k!}S_{n,k}[\psi]
$$
with
\begin{equation}\label{def:Snk}
S_{n,k}[\psi]=S_{n,k}[\psi,\mu]\deff \int_{\K^k}\det\left(\sqrt{\psi(u_j)}\msf K_n(u_j,u_\ell)\sqrt{\psi(u_\ell)}\right)_{j,\ell=1}^k \prod_{j=1}^k \dd \mu(u_j).
\end{equation}
Similarly, if $\phi:\K \to[0,+\infty)$ is such that  $\sqrt\phi\mathbf K\sqrt\phi$ is trace-class on $L^2(\mu)$, 
$$
\det\left(\itid-\sqrt\phi\mathbf K\sqrt\phi\right)_{L^2(\mu)}=\sum_{k=0}^\infty\frac{(-1)^k}{k!}S_{k}[\phi],
$$
where analogously
\begin{equation}\label{def:Sk}
S_{k}[\phi]=S_k[\phi,\mu]\deff \int_{\K^k}\det\left(\sqrt{\phi(u_j)}\msf K(u_j,u_\ell)\sqrt{\phi(u_\ell)}\right)_{j,\ell=1}^k \prod_{j=1}^k \dd \mu(u_j).
\end{equation}
We will prove first that $S_{n,k}[\psi_n]\to S_k[\phi]$ as $n\to\infty$ for any $k$, with $\psi_n=\sigma_n$, $\phi=\sigma$, and also with $\psi_n=\sigma_n+h-\sigma_n h$, $\phi=\sigma+h-\sigma h$. Afterwards we will prove that the whole series $\sum_{k=0}^\infty\frac{(-1)^k}{k!}S_{n,k}[\psi_n]$ 
converges as $n\to\infty$ to $\sum_{k=0}^\infty\frac{(-1)^k}{k!}S_{k}[\phi]$ for the same choices of $\psi_n,\phi$. These results are proven in Lemmas~\ref{lemma:SB} and \ref{lemma:S2}, respectively, and imply \eqref{eq:limittoprove1B}--\eqref{eq:limittoprove2B}. For proving them, we will need two versions of Hadamard's inequality, which we recall as the next lemma.

\begin{lemma}
    For a $k\times k$ matrix $M$, denote its $j$-th column by $M_j \in \C^k$, and for a vector $v=(v_1,\cdots, v_k)^T\in \C^k$ denote
    $$
    \|v\|_2\deff \left(\sum_{i=1}^k |v_i|^2 \right)^{1/2},\quad \|v\|_\infty \deff \sup_{1\leq i\leq k} |v_i|.
    $$
    Then for any $k\times k$ matrices $M$ and $L$, the inequalities
\begin{equation}\label{Had1}
    |\det M|\leq \prod_{i=1}^k \|M_i\|_2
\end{equation}
    and
\begin{equation}\label{Had2}
    |\det M - \det L|\leq k^{k/2}\left(\max_{1\leq i\leq k} (\|M_i\|_\infty, \|L_i\|_\infty) \right)^{k-1}\max_{1\leq i\leq k} \|M_i-L_i\|_\infty
\end{equation}
    hold true.
\end{lemma}
\begin{proof}
Inequality \eqref{Had1} is the classical Hadamard's inequality.
{\rev Inequalities similar to \eqref{Had2} are known for other matrix norms, see for instance \cite{IpsenRehmanDetBounds}, in particular Theorem~2.12 therein. But we failed to find it in the exact form \eqref{Had2}, and for completeness we provide a simple proof}. Let $(v_1,\hdots, v_k)$ be the matrix whose $i$-th column is the column vector {$v_i\in \C^k$,} and $M_j$ be the $j$-th column of the matrix $M$. By linearity of the determinant on each column,
\begin{align*}
\det M & =\det (M_1-L_1,M_2,M_3,\cdots, M_k)+\det(L_1,M_2,M_3,\cdots, M_k) \\
 & = \begin{multlined}[t]
 \det (M_1-L_1,M_2,M_3,\cdots, M_k)+\det(L_1,M_2-L_2,M_3,\cdots, M_k) \\ +\det(L_1,L_2,M_3,\cdots, M_k)
 \end{multlined} \\
& = \sum_{i=1}^k \det(L_1,\cdots, L_{i-1},M_i-L_i,M_{i+1},\cdots, M_k) + \det L.
\end{align*}
Subtracting $\det L$ from both sides, and using \eqref{Had1} and the basic inequality $\|v\|_2\leq k^{1/2}\|v\|_\infty$ several times, inequality \eqref{Had2} follows.
\end{proof}

\begin{lemma}\label{lemma:SB}
Let $k\in\mathbb N$, and assume that Assumption \ref{assumption:sigma} holds. 
Then,
\[\lim_{n\to\infty}S_{n,k}[\sigma_n]=S_k[\sigma],\qquad \lim_{n\to\infty}S_{n,k}[\sigma_n+h-\sigma_n h]=S_k[\sigma+h-\sigma h].\]
\end{lemma}

\begin{proof}
It suffices to verify the second limit, as the first limit follows from it by setting $h\equiv 0$.

Write $\phi_n=\sigma_n+h-\sigma_nh$ and $\phi=\sigma+h-\sigma h$. We have
\[
S_{n,k}[\phi_n]=\int_{\K^k}\det\left(\sqrt{\phi_n(u_j)}\msf K_n(u_j,u_\ell)\sqrt{\phi_n(u_\ell)}\right)_{j,\ell=1}^k \prod_{j=1}^k\dd \mu(u_j).\]
By Assumption \ref{assumption:sigma}, the integrand converges point-wise to 
$$
\det\left(\sqrt{\phi(u_j)}\msf K(u_j,u_\ell)\sqrt{\phi(u_\ell)}\right)_{j,\ell=1}^k
$$ 
for $\mu$-a.e. $(u_1,\ldots, u_k)\in\K^k$. Moreover, with $\Phi,\Psi$ as in Assumption \ref{assumption:sigma} for a bounded set $F$ containing $\supp h$, set 
$$
Z\deff \{(x_1,\cdots, x_k)\in \K^k\mid \Phi(x_j)\Psi(x_j)=0, \text{ for some }j\}.
$$
The integrand of $S_{n,k}$ vanishes on $Z$, and we may write
\begin{equation}\label{eq:boundSnkB}
S_{n,k}[\phi_n]= \int_{\K^k\setminus Z}\det\left(\frac{\sqrt{\phi_n(u_j)}\msf K_n(u_j,u_\ell)\sqrt{\phi_n(u_\ell)}}{\Phi(u_j)\Psi(u_\ell)}\right)_{j,\ell=1}^k \prod_{j=1}^k\Phi(u_j)\Psi(u_j)\dd \mu(u_j).
\end{equation}
Thanks to the inequalities $0\leq \phi_n=\sigma_n(1-h)+h\leq \sigma_n+\itid_F$, each entry of the matrix is bounded in absolute value by $1$, hence the determinant is bounded in $n$, for fixed $k$. It follows that there exists a constant $C_k>0$ such that the absolute value of the integrand is bounded from above by
\[ C_k\prod_{j=1}^k\Phi(u_j)\Psi(u_j),\]
which is integrable over $\K^k$, since $\Phi,\Psi\in L^2(\mu)$.
Thus, by Lebesgue's Dominated Convergence Theorem, we have that
$$
\lim_{n\to\infty}{S_{n,k}[\phi_n]}=\int_{\K^k}\det\left(\sqrt{\phi(u_j)}\msf K(u_j,u_\ell)\sqrt{\phi(u_\ell)}\right)_{j,\ell=1}^k \prod_{j=1}^k\dd \mu(u_j)=S_k[\phi],
$$
concluding the proof.
\end{proof}

\begin{lemma}\label{lemma:S2}
If Assumption \ref{assumption:sigma} holds, then we have the limits \eqref{eq:limittoprove1B} and \eqref{eq:limittoprove2B}.
\end{lemma}

\begin{proof}
As before, it is sufficient to prove \eqref{eq:limittoprove1B}, the limit \eqref{eq:limittoprove2B} then follows setting $h\equiv 0$
We keep denoting $\phi_n=\sigma_n +h-\sigma_n h $ and $\phi=\sigma+h-\sigma h$.
Recall that 
\[
\det\left(\itid-\sqrt{\phi_n}\mathbf K_n\sqrt{\phi_n}\right)=\sum_{k=0}^\infty\frac{(-1)^k}{k!}S_{n,k}[\phi_n],
\] 
and we know from Lemma \ref{lemma:SB} that $S_{n,k}[\phi_n]\to S_k[\phi]$ for every $k\in\mathbb N$ as $n\to\infty$.
In order to use Lebesgue's Dominated Convergence Theorem for series, we bound $S_{n,k}[\phi_n]$.
We start again from identity \eqref{eq:boundSnkB}, in which we have the determinant of a $k\times k$ matrix with entries bounded in absolute value by $1$.
By Hadamard's inequality \eqref{Had1}, the determinant is bounded in absolute value by $k^{k/2}$, and we obtain
\[
\left|S_{n,k}[\phi_n]\right|\leq k^{k/2}\left(\int_{\K }\Phi(u)\Psi(u)\dd \mu(u)\right)^k\leq k^{k/2}\|\Phi\|_2^k\|\Psi\|_2^k,
\]
where we used the Cauchy-Schwarz inequality. But the series 
\[
\sum_{k=0}^\infty\frac{k^{k/2}}{k!}\|\Phi\|_2^k\|\Psi\|_2^k
\]
is convergent, so we can indeed use Lebesgue's Dominated Convergence Theorem for series to obtain
\begin{align*}
\lim_{n\to\infty}\det\left(\itid-\sqrt{\phi_n}\mathbf K_n\sqrt{\phi_n}\right) & =\lim_{n\to\infty}\sum_{k=0}^\infty\frac{(-1)^k}{k!}S_{n,k}[\phi_n]
=\sum_{k=0}^\infty\frac{(-1)^k}{k!}S_{k}[\phi] \\ 
& =\det\left(\itid-\sqrt\phi\mathbf K\sqrt\phi\right),
\end{align*}
and \eqref{eq:limittoprove1B} is proved.

\end{proof}

\section{Proof of Theorem \ref{thm:biOPE-varying}}\label{sec:proofvarying}

We now consider biorthogonal ensembles $\mathcal X_n$ with kernels $\msf K_n$ which satisfy Assumption \ref{assumption:sigmavarying}. In particular, the kernel $\msf K_n$ acts on a space $L^2(\mu_n)$ with $n$-dependent Radon measure $\mu_n$, and the sequence of measures $(\mu_n)$ converges to a measure $\mu$ in the sense of Assumption~\ref{assumption:sigmavarying} (3). As before, we consider the deformed biorthogonal ensemble $\mathcal X_n^{\sigma_n}$ from \eqref{def:deform}, with probability generating functional $\mathcal G_n^{\sigma_n}$, and we assume that $h:\K \to [0,+\infty)$ is continuous and compactly supported, with $\|h\|_\infty <1$

Our strategy is the same as in Section~\ref{sec:proofB}. Thanks to Lemma~\ref{lemma:GB}, it suffices to prove
\begin{align}
&\label{eq:limittoprove1C}
\begin{multlined}[b]
\lim_{n\to\infty} \det\left(\itid-\sqrt{\sigma_n+h-\sigma_n h}\mathbf K_n\sqrt{\sigma_n+h-\sigma_n h}\right)_{L^2(\mu_n)} =\\ 
\det\left(\itid-\sqrt{\sigma+h-\sigma h}\mathbf K\sqrt{\sigma+h-\sigma h}\right)_{L^2(\mu)},
\end{multlined}
\\
&\label{eq:limittoprove2C}
\lim_{n\to\infty}\det\left(\itid-\sqrt{\sigma_n}\mathbf K_n\sqrt{\sigma_n}\right)_{L^2(\mu_n)}=\det\left(\itid-\sqrt\sigma\mathbf K\sqrt\sigma\right)_{L^2(\mu)},
\end{align}
as in that case we obtain again that
\begin{equation}
\lim_{n\to\infty}\mathcal G_n^{\sigma_n}[h]=\frac{\det\left(1-\sqrt{\sigma+h-\sigma h}\mathbf K\sqrt{\sigma+h-\sigma h}\right)_{L^2(\mu)}}{\det\left(1-\sqrt\sigma\mathbf K\sqrt\sigma\right)_{L^2(\mu)}}=\mathcal G^\sigma[h].
\end{equation}
Recall the quantities $S_{n,k}^\mu[\psi]$ and $S_k^\mu[\phi]$ that were introduced in \eqref{def:Snk} and \eqref{def:Sk}, where we now make the dependence on the measure $\mu$ explicit in our notation. In turn, to prove \eqref{eq:limittoprove1C}--\eqref{eq:limittoprove2C} it suffices to prove 
\begin{align}
& \lim_{n\to \infty}\sum_{k=1}^\infty\frac{(-1)^k}{k!} S_{n,k}^{\mu_n}[\sigma_n]= \sum_{n=1}^\infty \frac{(-1)^k}{k!}S_k^\mu[\sigma] \qquad \text{and} \label{eq:serieslimitmun1} \\
& \lim_{n\to\infty} \sum_{k=1}^\infty \frac{(-1)^k}{k!}S_{n,k}^{\mu_n}[\sigma_n+h-\sigma_n h]= \sum_{n=1}^\infty \frac{(-1)^k}{k!}S_k^\mu[\sigma+h-\sigma h]. \label{eq:serieslimitmun2}
\end{align}
We will show these series convergences by first showing that each term converges to the corresponding term in the limit, and then arguing that the limit commutes with the sum.
We first need to establish a technical result.

\begin{lemma}\label{lem:weakconvergence}
Let $(\nu_n)$ be a sequence of finite Borel measures over a fixed Euclidean space $E$, converging weakly to a finite Borel measure $\nu$ on $E$ as $n\to\infty$. 
\begin{enumerate}[(i)]
    \item If $(f_n)$ is a sequence of uniformly bounded Borel measurable functions and $f$ is a bounded continuous function, for which
    $$
  \lim_{n\to\infty}  \|\itid_B(f_n-f)\|_{L^\infty(\nu_n)}= 0,
    $$
    for any compact $B\subset \K $, then
$$
\lim_{n\to\infty}\int f_n \, \dd\nu_n=\int f\, \dd\nu.
$$

\item If $G$ is an open set for which $\nu(\partial G)=0$, then $\restr{\nu_n}{G}\stackrel{*}{\to}\restr{\nu}{G}$ as $n\to\infty$.
\end{enumerate}

\end{lemma}
\begin{proof}
Write
$$
\int f_n \, \dd\nu_n- \int f\, \dd\nu=\int (f_n-f) \, \dd\nu_n + \int f\, \dd\nu_n-\int f\,\dd\nu.
$$
From the weak convergence and the continuity of $f$, we obtain that $\int f\, \dd\nu_n-\int f\,\dd\nu\to 0$ as $n\to \infty$, and we now prove that $\int(f-f_n)\dd\nu_n\to 0$.

The uniform boundedness of the sequence $(f_n)$ and the boundedness of $f$ imply that for some constant $M>0$, the bound $|f_n(x)-f(x)|\leq M$ holds true for every $x\in E $ and every $n$. 

Fix an arbitrary $\varepsilon>0$. Since $\nu_n\stackrel{*}{\to}\nu$ and $\nu$ is finite, there is a compact $B\subset E$ for which $\nu_n(E \setminus B)<\varepsilon/M$ for every $n$ sufficiently large. We then write
\begin{align*}
\left|\int (f_n-f)\,\dd\nu_n\right| & \leq \int_{E \setminus B}|f_n-f|\, \dd\nu_n+\int_B |f_n-f|\, \dd\nu_n\\
& \leq \varepsilon+\|(f_n-f)1_B\|_{L^\infty(\nu_n)} \nu_n(B) \\
& \leq \varepsilon+\|(f_n-f)1_B\|_{L^\infty(\nu_n)} \nu_n(E).
\end{align*}

The convergence $\nu_n\stackrel{*}{\to}\nu$ implies that $\nu_n(E)=\int \itid \, \dd\nu_n\to \nu(E)$, so the last term on the right-hand side above converges to zero by assumption. Since $\varepsilon>0$ is arbitrary, the proof of (i) is complete.
{Part (ii) is folklore, but we have not been able to find a detailed reference, so we provide a proof of it. We start by recalling Portmanteau's Theorem \cite[Theorem~2.1]{BillingsleyBookConv}, which says that for Borel {\it probability} measures $\lambda_n$ and $\lambda$ the following conditions are equivalent:
\begin{enumerate}[(1)]
    \item $\lambda_n\stackrel{*}{\to}\lambda$;
    \item $\liminf_{n\to\infty}\lambda_n(U)\geq \lambda(U)$, for every open set $U\subset E$;
    \item $\limsup_{n\to\infty}\lambda_n(F)\leq \lambda(F)$, for every closed set $F\subset E$;
    \item $\lambda_n(B)\to \lambda(B)$, for every Borel set $B$ for which $\lambda(\partial B)=0$.
\end{enumerate}

To prove (ii), assume first that $\nu_n$ and $\nu$ are probability measures, with $\nu_n\stackrel{*}{\to}\nu$, and $G$ is an open set for which $\nu(\partial G)=0$. If $\nu(G)=0$, then $\restr{\nu}{G}\equiv 0$, and 
$$
\limsup \nu_n(G)\leq \limsup \nu_n(\overline G)\leq \nu(\overline G)=\nu(G)\cup \nu(\partial G)=0.
$$
and by the implication $(3)\Rightarrow (1)$ we see that $\restr{\nu_n}{G}\to 0=\restr{\nu}{G}$. Assuming now that $\nu(G)> 0$, the weak convergence $\nu_n\stackrel{*}{\to}\nu$ and (4) above imply that $\nu_n(G)\to \nu(G)>0$, and for $n$ sufficiently large we may introduce
$$
\lambda_n\deff \frac{1}{\nu_n(G)}\restr{\nu_n}{G}\quad \text{and}\quad \lambda\deff \frac{1}{\nu(G)}\restr{\nu}{G}.
$$
We already have the convergence $\nu_n(G)\to \nu(G)$, so to conclude that $\restr{\nu_n}{G}\stackrel{*}{\to}\restr{\nu}{G}$ it is sufficient to verify that $\lambda_n\stackrel{*}{\to}\lambda$. For the latter, we now verify the condition (2): given any open set $U$, 
\begin{align*}
\liminf \lambda_n(U) & \geq \lim \frac{1}{\nu_n(G)} \liminf \nu_n(G\cap U)=\frac{1}{\nu(G)}\liminf \nu_n(G\cap U)\\ 
& \geq \frac{1}{\nu(G)}\nu(G\cap U)=\restr{\lambda}{G}(U).
\end{align*}
where the last inequality follows from the weak convergence of probability measures $\nu_n\stackrel{*}{\to}\nu$. This concludes the proof of (ii) when $\nu_n$, $\nu$ are probability measures.

For the general case of finite measures, if $\nu(E)=0$, that is, $\nu$ is the trivial measure, then (ii) is immediate. Otherwise, the weak convergence $\nu_n\stackrel{*}{\to}\nu$ gives that
$$
\nu_n(E)=\int \itid_E\, \dd\nu_n\to \int \itid_E \,\dd\nu,
$$
and we may assume that $\nu_n(E)>0$. We now introduce
$$
\wh{\nu}_n\deff \frac{1}{\nu_n(E)}\nu_n,\quad \wh{\nu}\deff \frac{1}{\nu(E)}\nu=\nu(E),
$$
which are probability measures, and as such the first part of the proof yields that $\restr{\wh\nu_n}{G}\stackrel{*}{\to}\restr{\wh\nu}{G}$. But $\restr{\nu_n}{G}=\nu_n(E)\restr{\wh\nu_n}{G}$, $\restr{\nu}{G}=\nu(E)\restr{\wh\nu}{G}$ and $\nu_n(E)\to \nu(E)$, and the result follows.}
\end{proof}

\begin{lemma}\label{lemma:S(B)}
Let $k\in\mathbb N$, and let Assumption \ref{assumption:sigmavarying} hold. 
Then,
\[\lim_{n\to\infty}S_{n,k}^{\mu_n}[\sigma_n]=S_k^\mu[\sigma],\qquad \lim_{n\to\infty}S_{n,k}^{\mu_n}[\sigma_n+h-\sigma_n h]=S_k^\mu[\sigma+h-\sigma h].\]
\end{lemma}
\begin{proof}
For any bounded functions $\phi_n$ and $\phi$, we have
\begin{multline}\label{prop:Snkineq1}
|S_{n,k}^{\mu_n}[\phi_n]-S_k^\mu[\phi]|\leq 
\int_{\K^k} \Delta_n(x_1,\ldots,x_k) \prod_{j=1}^k \dd\mu_n(x_j)\\ +\left|\int_{\K^k} \msf D(x_1,\ldots,x_k) \prod_{j=1}^k \dd\mu_n(x_j)-\int_{\K^k}\msf D(x_1,\ldots,x_k) \prod_{j=1}^k \dd\mu(x_j)\right|,
\end{multline}
where
\begin{align*}
& \msf D(x_1,\ldots,x_k )\deff \det\left(\sqrt{\phi(x_j)}\msf K(x_j,x_\ell)\sqrt{\phi(x_\ell)}\right)_{j,\ell=1}^k,\\
& 
\begin{multlined}[t]
\Delta_n(x_1,\ldots,x_k)\deff  
\left| \det\left(\sqrt{\phi_n(x_j)}\msf K_n(x_j,x_\ell)\sqrt{\phi_n(x_\ell)}\right)_{j,\ell=1}^k-\det\left(\sqrt{\phi(x_j)}\msf K(x_j,x_\ell)\sqrt{\phi(x_\ell)}\right)_{j,\ell=1}^k\right|.
\end{multlined}
\end{align*}
As before, we are interested in the choices $(\phi_n,\phi)=(\sigma_n,\sigma)$ and $(\phi_n,\phi)=(\sigma_n+h-\sigma_n h,\sigma+h-\sigma h)$ with $h:\K \to [0,1)$ continuous and compactly supported. The former case follows from the latter by setting $h\equiv 0$. So we focus only on the latter, and verify that each of the two terms on the right-hand side of \eqref{prop:Snkineq1} goes to zero.

{Let $\Phi,\Psi$ be as in Assumption~\ref{assumption:sigmavarying}, where $F$ is a bounded set for which $\supp h\subset F$ (if $h\equiv 0$, we can choose for instance $F=\emptyset$). Set
$$
G\deff \K^k\setminus \{(x_1,\ldots, x_k)\in \K^k\mid  \Phi(x_j)\Psi(x_j)=0 \text{ for some }j\}.
$$

Because the functions $\Phi$ and $\Psi$ are continuous, the set $G$ is open. Assumption~\ref{assumption:sigmavarying} (3) implies that whenever $(x_1,\cdots,x_k)\in \K^k\setminus G$, both determinants in $\Delta_n(x_1,\ldots,x_k)$ vanish because at least one row in the matrices is identically zero, so that $\Delta_n(x_1,\ldots,x_k)=0$. Therefore, we may write the second line of \eqref{prop:Snkineq1} as
\begin{equation}\label{eq:var_aux1}
\int_G \frac{\msf D(x_1,\ldots,x_k)}{\prod_{j=1}^k \Phi(x_j)\Psi(x_j)} \prod_{j=1}^k\Phi(x_j)\Psi(x_j)\dd(\mu_n-\mu)(x_j).
\end{equation}
as well as 
\begin{equation}\label{eq:var_aux2}
\begin{aligned}
\int_{\K^k} \Delta_n(x_1,\ldots,x_k) \prod_{j=1}^k \dd\mu_n(x_j) & = \int_{G} \frac{\Delta_n(x_1,\ldots,x_k)}{\prod_{j=1}^k \Phi(x_j)\Psi(x_j)} \prod_{j=1}^k \Phi(x_j)\Psi(x_j)\dd\mu_n(x_j).
\end{aligned}
\end{equation}

The function $\frac{\msf D(x_1,\ldots,x_k)}{\prod_{j=1}^k \Phi(x_j)\Psi(x_j)}$ is continuous on the open set $G$, and by the definition of $G$ we know that the measure of $\partial G$ under $\prod {\Phi}(x_j)\Psi(x_j)\dd\mu(x_j)$ is zero. The weak convergence $\Phi\Psi\dd\mu_n\stackrel{*}{\to}\Phi\Psi\dd\mu $ combined with Lemma~\ref{lem:weakconvergence} (ii) then implies that the difference of integrals in \eqref{eq:var_aux1} goes to $0$.}

To conclude the proof it remains to verify that the integral on the right-hand side of \eqref{eq:var_aux2} converges to $0$. For that, we will use Lemma~\ref{lem:weakconvergence} (i) with $E=\K^k$, an arbitrary compact set $B$, and the choices
\begin{align*}
& f_n(x_1,\ldots, x_k)=\itid_{G}(x_1,\ldots,x_k)\frac{\Delta_n(x_1,\ldots,x_k)}{\prod_{j=1}^k \Phi(x_j)\Psi(x_j)},&&  \dd\nu_n(x)=\prod_{j=1}^k \Phi(x_j)\Psi(x_j) \dd\mu_n(x_j), \\
& f(x_1,\ldots, x_k)=0, && \dd\nu(x)=\prod_{j=1}^k \Phi(x_j)\Psi(x_j) \dd\mu(x_j).
\end{align*}
From Assumption~\ref{assumption:sigmavarying} (3) we have $\nu_n\stackrel{*}{\to}\nu$. Using linearity of determinants,
\begin{multline*}
\frac{\Delta_n(x_1,\ldots,x_k)}{\prod_{j=1}^k \Phi(x_j)\Psi(x_j)}=\\ 
 \left| \det\left(\frac{\sqrt{\phi_n(x_j)}\msf K_n(x_j,x_\ell)\sqrt{\phi_n(x_\ell)}}{{\Phi(x_j)\Psi(x_\ell)}}\right)_{j,\ell=1}^k-\det\left(\frac{\sqrt{\phi(x_j)}\msf K(x_j,x_\ell)\sqrt{\phi(x_\ell)}}{\Phi(x_j)\Psi(x_\ell)}\right)_{j,\ell=1}^k\right|
\end{multline*}
Since $\supp h\subset F$ we have
$$
0\leq \phi_n = \sigma_n(1-h)+h\leq \sigma_n+\itid_F,
$$
and therefore Assumption~\ref{assumption:sigmavarying} (3) implies that each entry of each determinant in $\Delta_n$ is bounded by $1$. Applying the first Hadamard's inequality \eqref{Had1}, we obtain
\begin{equation}\label{prop:Snkineq2}
\frac{\Delta_n(x_1,\ldots,x_k)}{\prod_{j=1}^k \Phi(x_j)\Psi(x_j)} \leq 2k^{k/2},\quad (x_1,\hdots,x_k)\in G,
\end{equation}
so the sequence $(f_n)$ is indeed uniformly bounded. 

A simple calculation shows that
\begin{equation}\label{}
\|\itid_B f_n \|_{L^\infty(\nu_n)}=\| \itid_{B\cap G}\, \Delta_n \|_{L^\infty(\otimes^k \mu_n)}.
\end{equation}
To bound the right-hand side, we now show how to apply the second Hadamard inequality. With
\begin{align*}
& M=M(x)\deff \itid_{G\cap B}(x)\left(\sqrt{\phi_n(x_j)} \msf K_n(x_j,x_\ell)\sqrt{\phi_n(x_\ell)}\right)_{j,\ell=1}^k,\\
& L=L(x)\deff \itid_{G\cap B}(x)\left(\sqrt{\phi(x_j)} \msf K(x_j,x_\ell)\sqrt{\phi(x_\ell)}\right)_{j,\ell=1}^k,
\end{align*}
we have $\itid_{G\cap B}\Delta_n=|\det M-\det L|$ pointwise, and Assumption~\ref{assumption:sigmavarying} (3) gives the inequalities
$$
\|M_j(x)\|_\infty, \|L_j(x)\|_\infty \leq \|\Phi\|_{L^{\infty}(\mu_n)}\|\Psi\|_{L^{\infty}(\mu_n)},\quad x\in G\cap B.
$$
From the definition of $M_j$ and $L_j$ we also have for each $j$ that 
\begin{align*}
\|M_j(x)-L_j(x)\|_\infty \leq & \; \|\itid_B(\sqrt{\phi_n}-\sqrt{\phi})\|_{L^\infty(\mu_n)}\|\itid_B\otimes \itid_B\msf K_n\|_{L^\infty(\mu_n\otimes \mu_n)}\|\itid_B\sqrt{\phi}\|_{L^\infty(\mu_n)} \\
& +\|\itid_B\sqrt{\phi}\|_{L^\infty(\mu_n)}\|\itid_B\otimes \itid_B(\msf K_n-\msf K)\|_{L^\infty(\mu_n\otimes \mu_n)}\|\itid_B\sqrt{\phi}\|_{L^\infty(\mu_n)}
\\
& + \|\itid_B\sqrt{\phi}\|_{L^\infty(\mu_n)}\|\itid_B\otimes \itid_B \msf K \|_{L^\infty(\mu_n\otimes \mu_n)}\|\itid_B(\sqrt{\phi_n}-\sqrt{\phi})\|_{L^\infty(\mu_n)}.
\end{align*}
Denote the right-hand side of this last inequality by $\rho_n$. Hadamard's inequality \eqref{Had2} then yields that
$$
\|\itid_{B\cap G}\Delta_n\|_{L^\infty(\otimes^k \mu_n)}\leq k^{k/2}  (\|\Phi\|_{L^{\infty}(\mu_n)})^{k-1}(\|\Psi\|_{L^{\infty}(\mu_n)})^{k-1} \rho_n.
$$
From Assumption~\ref{assumption:sigmavarying} (3), the $L^\infty$ norms of $\Phi$ and $\Psi$ on the right-hand side above remain bounded as $n\to \infty$. From the inequalities $0\leq \sqrt{u+1}-\sqrt{v+1}\leq \sqrt{2}(\sqrt{u}-\sqrt{v})$ for $0\leq v\leq u\leq 1$ applied to $u=\max\{\sigma_n,\sigma\}, v=\min\{\sigma_n,\sigma\}$, we obtain that 
$$
|\sqrt{\phi_n}-\sqrt{\phi}|\leq |\sqrt{\sigma_n+\itid_F}-\sqrt{\sigma+\itid_F}|\leq \sqrt{2}|\sqrt{\sigma_n}-\sqrt{\sigma}|,
$$
and when combined with Assumption~\ref{assumption:sigmavarying} (2) it gives that $\rho_n\to 0$. Hence, $\|\itid_B f_n\|_{L^\infty(\nu_n)}\to 0$ for every bounded set $B$, Lemma~\ref{lem:weakconvergence} (i) is applicable, and we conclude that the term on the right-hand side of the first line of \eqref{prop:Snkineq1} goes to $0$, as we wanted.
\end{proof}

\begin{remark}
Note that in many situations, one may take $\Phi,\Psi$ strictly positive, such that $G=\K^k$ in the above proof. This simplifies some of the arguments such that the proof can be shortened considerably. In particular, we do not need Lemma \ref{lem:weakconvergence} (ii) in that case.
\end{remark}

\begin{lemma}\label{lemma:SB2}
If Assumption~\ref{assumption:sigmavarying} holds, then the limits \eqref{eq:serieslimitmun1}--\eqref{eq:serieslimitmun2} hold true.
\end{lemma}
\begin{proof}
Similarly to Lemma~\ref{lemma:S(B)}, it suffices to prove \eqref{eq:serieslimitmun2}, the limit \eqref{eq:serieslimitmun1} then follows by setting $h\equiv 0$. We follow closely the notation and ideas from the proof of Lemma~\ref{lemma:S(B)}. 

Set as usual $\phi_n=\sigma_n+h-\sigma_nh$, $\phi=\sigma+h-\sigma h$ and
$$
\msf D_n(x_1,\ldots,x_k)\deff \det \left(\sqrt{\phi_n(x_j)}\msf K_n(x_j,x_\ell)\sqrt{\phi_n(x_\ell)}\right)_{j,\ell=1}^k.
$$
Proceeding as we did for \eqref{prop:Snkineq2}, we obtain
\begin{multline*}
|S_{n,k}^{\mu_n}[\sigma_n]| 
\begin{aligned}[t]
& \leq \int_{G} \left| \frac{\msf D_n(x_1,\ldots,x_k)}{\prod_{j=1}^k \Phi(x_j)\Psi(x_j)} \right| \prod_{j=1}^k \Phi(x_j)\Psi(x_j)\dd\mu_n(x_j) \leq k^{k/2} \int \prod_{j=1}^k \Phi(x_j)\Psi(x_j)\dd\mu_n(x_j) \\
& \leq k^{k/2} \|\Phi\|_{L^2(\mu_n)}^k\|\Psi\|_{L^2(\mu_n)}^k,
\end{aligned}
\end{multline*}
where we used Cauchy-Schwarz for the last inequality. By assumption, the $L^2$ norms on the right-hand side remain bounded as $n\to \infty$, yielding that the series
$$
\sum_{k=1}^\infty \frac{k^{k/2}}{k!} \|\Phi\|_{L^2(\mu_n)}^k\|\Psi\|_{L^2(\mu_n)}^k
$$
is convergent. Therefore, Lebesgue's Dominated Convergence Theorem applies to the series
$$
\sum_{k=0}^\infty \frac{1}{k!}S_{n,k}^{\mu_n}[\phi_n]
$$
and it provides \eqref{eq:serieslimitmun2}.
\end{proof}
This concludes the proof of Theorem \ref{thm:biOPE-varying}.


\subsection*{Acknowledgements}
TC acknowledges support by {\em FNRS Research
Project T.0028.23} and by the {\em Fonds Sp\'ecial de Recherche} of UCLouvain. GS acknowledges support by São Paulo Research Foundation (FAPESP) under Grants \# 2019/16062-1 and \# 2020/02506-2, by Brazilian National Council for Scientific and Technological Development (CNPq) under Grant \# 306183/2023-4, and by the {\it Programa de Apoio aos Novos Docentes} PRPI-USP.

\bibliographystyle{abbrv}  






\end{document}